\documentclass[12pt]{amsart}
\usepackage{amssymb}
\usepackage{graphicx}
\usepackage[cmtip,all]{xy}
\usepackage{bbm}
\usepackage{comment}
\usepackage{mathrsfs}
\usepackage{mathtools}
\usepackage{microtype}
\usepackage{xcolor}
\usepackage[dvipdfmx]{hyperref}
\usepackage{autobreak}
\hypersetup{
 colorlinks,
 linkcolor={blue!50!black},
 citecolor={blue!50!black},
 urlcolor={blue!80!black}
}

\calclayout

\def\DefineSymbol#1#2{\newcommand{#1}{{\mathrm {#2}}}}
\def\DefineCategory#1#2{\newcommand{#1}{{\mathrm {#2}}}}

\theoremstyle{plain}
	\newtheorem{theorem}{Theorem}[section]
	\newtheorem{lemma}[theorem]{Lemma}
	\newtheorem{proposition}[theorem]{Proposition}
	\newtheorem{corollary}[theorem]{Corollary}
\theoremstyle{definition}
	\newtheorem{definition}[theorem]{Definition}
	\newtheorem{lemma and definition}[theorem]{Lemma and Definition}
	\newtheorem{example}[theorem]{Example}

	\newtheorem{construction}[theorem]{Construction}
	
\theoremstyle{remark}
	\newtheorem{remark}[theorem]{Remark}
	\numberwithin{equation}{section}

\DefineSymbol{\pr}{pr}
\DefineSymbol{\id}{id}
\DefineSymbol{\const}{const}
\DefineSymbol{\op}{op}
\DefineSymbol{\diag}{diag}
\DefineSymbol{\proet}{pro\acute{e}t}
\DefineSymbol{\cond}{cond}
\DefineSymbol{\conti}{conti}
\DefineSymbol{\Cond}{Cond}
\DefineSymbol{\dCond}{dCond}
\DefineSymbol{\disc}{disc}
\DefineSymbol{\Tot}{Tot}
\DefineSymbol{\adic}{adic}
\DefineSymbol{\nc}{nc}
\DefineSymbol{\cpt}{cpt}

\DeclareMathOperator{\Hom}{Hom}
\DeclareMathOperator{\Ext}{Ext}
\DeclareMathOperator{\Fun}{Fun}

\DeclareMathOperator{\cofib}{cofib}

\DeclareMathOperator{\Spa}{Spa}

\DeclareMathOperator{\End}{End}
\DeclareMathOperator{\intHom}{\underline{Hom}}

\newcommand{\lra}{\longrightarrow}

\newcommand{\bs}{\blacksquare}

\newcommand{\hotimes}{\mathbin{\hat{\otimes}}}

\DefineCategory{\Set}{Set}
\DefineCategory{\Ab}{Ab}
\DefineCategory{\Ring}{Ring}
\DefineCategory{\Mod}{Mod}
\DefineCategory{\Alg}{Alg}
\DefineCategory{\Ch}{Ch}
\DefineCategory{\Mon}{Mon}
\DefineCategory{\CMon}{CMon}
\DefineCategory{\PCoh}{PCoh}
\DefineCategory{\Perf}{Perf}
\DefineCategory{\FP}{FP}
\DefineCategory{\cAff}{cAff}
\DefineCategory{\AnRing}{AnRing}
\DefineCategory{\Ani}{Ani}
\DefineCategory{\CAlg}{CAlg}

\newcommand{\Cat}{\mathcal{C}at}

\newcommand{\Cbb}{\mathbb{C}}

\newcommand{\Ebb}{\mathbb{E}}

\newcommand{\Lbb}{\mathbb{L}}

\newcommand{\Nbb}{\mathbb{N}}

\newcommand{\Zbb}{\mathbb{Z}}

\newcommand{\Acal}{\mathcal{A}}
\newcommand{\Bcal}{\mathcal{B}}
\newcommand{\Ccal}{\mathcal{C}}
\newcommand{\Dcal}{\mathcal{D}}
\newcommand{\Ecal}{\mathcal{E}}

\newcommand{\Mcal}{\mathcal{M}}
\newcommand{\Ncal}{\mathcal{N}}
\newcommand{\Ocal}{\mathcal{O}}

\newcommand{\Scal}{\mathcal{S}}

\newcommand{\mfrak}{\mathfrak{m}}
\newcommand{\nfrak}{\mathfrak{n}}

\renewcommand{\hat}{\widehat}
\renewcommand{\bar}{\overline}

\begin{document}

\title[faithfully flat descent]{Faithfully flat descent of quasi-coherent complexes on rigid analytic varieties via condensed mathematics}
\author{Yutaro Mikami}
\date{\today}
\address{Graduate School of Mathematical Sciences, University of Tokyo, 3-8-1 Komaba, Meguro-ku, Tokyo 153-8914, Japan}
\email{y-mikmi@g.ecc.u-tokyo.ac.jp}
\maketitle
\begin{abstract}
Faithfully flat descent of pseudo-coherent complexes in rigid geometry was proved by Mathew. 
In this paper, we generalize the result of Mathew to solid quasi-coherent complexes on rigid analytic varieties, which have been introduced by Clausen and Scholze by means of condensed mathematics.
\end{abstract}

\setcounter{tocdepth}{1}

\tableofcontents


\section*{Introduction}
We begin with the statement of faithfully flat descent in commutative algebra.
Let $A\to B$ be a faithfully flat map of (commutative unital) rings.
Faithfully flat descent states that an $A$-module can be described as a $B$-module with a \textit{descent datum}.
Recently Lurie and Mathew have generalized the theory of faithfully flat descent to complexes of modules in \cite{SAG} and \cite{Mat16} by using higher algebra.
This generalization is important even if we are interested only in static objects (i.e., not higher categorical objects).
For example, it is used in \cite{BS19,CS19} whose main theorems are concerned with static objects.

On the other hand, there is an analogue of the classical theory of faithfully flat descent in rigid geometry.
Let $K$ be a complete non-archimedean field, and let $A\to B$ be a map of affinoid $K$-algebras which is faithfully flat as a map of ordinary rings.
Then a coherent $A$-module can be described as a coherent $B$-module $N$ with an isomorphism $\tau\colon B \otimes_A N \cong N \otimes_A B$ of $B\otimes_A B$-modules which satisfies the cocycle condition (such $\tau$ is called a \textit{descent datum}); see \cite{BG98, Con03} for more details.
This result has been generalized by Mathew in \cite{Mat22} as follows:

\begin{theorem}[{\cite[Theorem 1.4]{Mat22}}]\label{thm:ff1}
 Let $K$ and $A \to B$ be  as above.
 Let $B^{n/A}$ denote the $(n+1)$-fold completed tensor product of $B$ over $A$.
Then we have an equivalence of $\infty$-categories
$$\PCoh(A) \overset{\sim}{\lra} \varprojlim_{[n] \in \Delta} \PCoh(B^{n/A}),$$
where $\PCoh(A)$ is the category of pseudo-coherent complexes of $A$-modules and $\Delta$ is the simplex category.
\end{theorem}

Unlike commutative algebra, in rigid geometry we only treat objects which satisfy some finiteness conditions, that is coherent modules and pseudo-coherent complexes.
This is because of issues of topological algebraic objects such as topological rings and topological modules.

Recently Clausen and Scholze have developed a new approach to treat such objects in \cite{CM,AG}, which is called \textit{condensed mathematics}. 
In condensed mathematics, topological spaces are generalized to \textit{condensed sets}, which are sheaves of sets on the site of profinite sets with the topology given by finite jointly surjective families of maps.
Similarly, topological groups, rings, etc.\  are generalized, respectively, to \textit{condensed groups, rings, etc}.\  
One of the advantages of using condensed objects is that the category of condensed abelian groups becomes an abelian category as opposed to the category of topological abelian groups.

In this context, complete affinoid pairs are generalized to \textit{analytic rings}. 
Roughly speaking, an analytic ring is a pair of a condensed ring $\Acal$ and some additional data $\Mcal$, which is called the \textit{functor of measures} of $(\Acal, \Mcal)$.
For an analytic ring $(\Acal, \Mcal)$, we can define a symmetric monoidal stable $\infty$-category $\Dcal(\Acal, \Mcal)$.
For a complete affinoid pair $(A,A^+)$, we can define an analytic ring $(A,A^+)_{\bs}$ associated to $(A,A^+)$, and the objects in $\Dcal((A, A^+)_{\bs})$ are called solid quasi-coherent complexes over $(A, A^+)$. 

The following is a basic property of solid quasi-coherent complexes.

\begin{theorem}[{\cite[Theorem 4.1]{And21}}]\label{thm:analytic descent}
Let $X$ be an analytic affinoid adic space and let $U$ denote an arbitrary open affinoid subspace of $X$.
Then the functor $U \mapsto \Dcal((\Ocal_X(U), \Ocal_X^+(U))_{\bs})$ defines a sheaf of $\infty$-categories on $X$.
\end{theorem}

In \cite{Mat22}, it is expected that faithfully flat descent in rigid geometry can be formulated in this framework.
In this paper, we will carry it out.
The following is the main theorem in this paper.

\begin{theorem}\label{thm:main}
 Let $K$ be a complete non-archimedean field, and let $A \to B$ be a faithfully flat map of affinoid $K$-algebras.
 Let $B^{n/A}$ denote the $(n+1)$-fold completed tensor product of $B$ over $A$.
Then we have an equivalence of $\infty$-categories
$$\Dcal((A,A^{\circ})_{\bs}) \overset{\sim}{\lra} \varprojlim_{[n] \in \Delta} \Dcal((B^{n/A},(B^{n/A})^{\circ})_{\bs}). $$
\end{theorem}

This theorem is a generalization of Theorem \ref{thm:ff1}. 
In the proof, we will consider formal models and prove faithfully flat descent for them.
We explain two difficulties appearing in the proof.

First, we briefly review a ``descendable property'' introduced by Mathew in \cite{Mat16}.
Let $f \colon A \to B$ be a map of discrete rings. 
We say that the map $f$ is \textit{descendable} if the pro-object $\{\Tot_{n}(B^{\bullet/A})\}_n$ of $\Dcal(A)$ is a constant pro-object which converges to $A$, where $B^{\bullet/A}$ is the \v{C}ech nerve of $A \to B$ and $\Tot_{n}(B^{\bullet/A})$ is the limit of the $n$-truncation of $B^{\bullet/A}$ (Definition \ref{descendable def}).
In \cite[Definition 3.18]{Mat16}, Mathew shows that finitely presented faithfully flat maps of discrete rings are descendable (\cite[Corollary 3.33]{Mat16}).
He also shows that for a descendable map $f \colon A\to B$ of discrete rings, there is an equivalence of $\infty$-categories
$$\Dcal(A) \overset{\sim}{\lra} \varprojlim_{[n] \in \Delta} \Dcal((B^{n/A})),$$
where $B^{n/A}$ is the $(n+1)$-fold derived tensor product of $B$ over $A$.
The key point of the proof of the latter claim is that the scalar extension functor $\Dcal(A) \to \Dcal(B)$ is given by the tensor product $- \otimes_A^{\Lbb} B$ in $\Dcal(A)$.

Now, we consider a map of affinoid $K$-algebras.
Let $K$ be a complete non-archimedean field and $\pi$ be a pseudo-uniformizer of $K$.
Let $A \to B$ be a faithfully flat map of admissible $\Ocal_K$-algebras (i.e., $\pi$-torsion-free topologically finitely presented $\Ocal_K$-algebras).
In this setting, the functor $- \otimes_{A_{\bs}}^{\Lbb} B_{\bs}$ is not equal to the functor $-\otimes_{A_{\bs}}^{\Lbb} \underline{B}$, where $\underline{B}$ is the object of $\Dcal(A_{\bs})$ associated to the topological $A$-module $B$.
Therefore, we cannot apply the argument in \cite{Mat16} naively.
However, we can prove the following theorem.

\begin{theorem}\label{nba}
There exists an object $N_{B/A} \in \Dcal(\underline{B})$ satisfying the following conditions:
\begin{itemize}
\item
The object $N_{B/A}$ is $A_{\bs}$-complete and compact as an object of $\Dcal(A_{\bs})$.
\item
We have an equivalence of functors from $\Dcal(A_{\bs})$ to $\Dcal(B_{\bs})$
\begin{align}
-\otimes_{A_{\bs}}^{\Lbb} B_{\bs} \overset{\sim}{\lra} R\intHom_{\underline{A}}(N_{B/A}, -).
\label{eq1}
\end{align}
\end{itemize}
In particular, the functor
$$- \otimes_{A_{\bs}}^{\Lbb} B_{\bs} \colon \Dcal(A_{\bs}) \to \Dcal(B_{\bs})$$
commutes with small limits.
Moreover we can construct $N_{B/A}$ and the equivalence (\ref{eq1}) functorially (for the precise statement, see Remark \ref{functoriality}).
\end{theorem}

By this theorem, we can make use of $N_{B/A}$ instead of $\underline{B}$ to prove faithfully flat descent.

Next, we want to reduce faithfully flat descent for admissible $\Ocal_K$-algebras to fppf descent for discrete rings by taking a quotient by an ideal of definition.
However, for an extremally disconnected set $S$, the map $A_{\bs}[S]\to \displaystyle \varprojlim_{n \geq 1} (A/\pi^n)_{\bs}[S]$ is not necessarily an equivalence, so we cannot reduce the problem to fppf descent for discrete rings naively.
To resolve this problem, we will introduce the notion of small complete adic rings. 
Roughly speaking, a small complete adic ring is a complete adic ring whose reduction is integral over some finitely generated $\Zbb$-algebra.
For a small complete adic ring, the map above becomes an equivalence.
By using a limit argument, we can reduce the problem to the case of small complete adic rings.
Then we can reduce the problem to the case over discrete rings by taking a quotient by an ideal of definition.

To complete the proof, it is necessary to prove fppf descent for discrete rings in the context of condensed mathematics.

\begin{theorem}\label{thm:main2}
Let $A \to B$ be a finitely presented faithfully flat map of discrete rings.
Let $B^{n/A}$ denote the $(n+1)$-fold tensor product of $B$ over $A$.
Then we have an equivalence of $\infty$-categories 
$$\Dcal(A_{\bs}) \overset{\sim}{\lra} \varprojlim_{[n] \in \Delta} \Dcal((B^{n/A})_{\bs}). $$
\end{theorem}

In this setting, we can also define $N_{B/A}$ as in Theorem \ref{nba}.
In order to prove Theorem \ref{thm:main2}, we will prove a descendable property of $N_{B/A}$ (Theorem \ref{colimit}) based on ideas of Mathew in \cite{Mat16}.
It is also necessary to prove uniform boundedness of $N_{B/A}$ (Proposition \ref{bdd}) for reducing faithfully flat descent for admissible $\Ocal_K$-algebras to fppf descent for discrete rings (see the proof of Theorem \ref{rig dual des}).

\begin{remark}
While writing this paper, the author was informed that Mann proves faithfully flat descent for discrete rings by a slightly different approach in his thesis \cite{Mann22}. 
In this paper, it will be proved by using $N_{B/A}$.
On the other hand, Mann considers a stable monoidal $\infty$-category $\Ecal(A_{\bs})$ instead of $\Dcal(A_{\bs})$.
By using it instead of $N_{B/A}$, Mann defines the notion of descendable maps of analytic rings and constructs a general theory of them to prove faithfully flat descent.
\end{remark}

\subsection*{Outline of the paper}
This paper is organized as follows. 
In Section 1, we begin with the definition of analytic rings, then recall the construction of analytic rings from affinoid pairs. For more details, see \cite{CM,AG,And21}.
In Section 2, we prove fppf descent for discrete rings in the context of condensed mathematics.
In Section 3, we introduce the notion of adic completeness in the context of condensed mathematics, then introduce the notion of small adic rings. 
In Section 4, we prove Theorem \ref{thm:main}.
Finally, in Section 5, we recover Theorem \ref{thm:ff1} from Theorem \ref{thm:main} by using the method used in \cite[Section 5]{And21}.

\subsection*{Convention}
\begin{itemize}
\item
All rings, including condensed ones, are assumed unital and commutative.
\item
For an $\infty$-category $\Ccal$, $0$-truncated objects of $\Ccal$ are called \textit{discrete objects} in \cite{HTT}.
However, this term conflicts with the term ``discrete" in the topological sense, so we use the term \textit{static object} to refer to an \textit{$0$-truncated object}.
\item 
In contrast to \cite{And21, Mann22}, we use the term \textit{ring} to refer to an \textit{ordinary ring} (not an animated ring). 
\item
All complete adic rings are assumed Hausdorff.
\item
We use the term ``extremally disconnected set" to refer to an \textit{extremally disconnected compact Hausdorff space}. 
\item
We use the terms \textit{f-adic ring} and \textit{affinoid pair} rather than \textit{Huber ring} and \textit{Huber pair}.
\item
We denote the simplex category by $\Delta$, which is the full subcategory of the category of totally ordered sets consisting of the totally ordered sets $[n]=\{0,\ldots,n\}$ for all $n \geq 0$. 
We also denote the subcategory of $\Delta$ with the same objects but where the morphisms are given by injective maps by $\Delta_s$.
Moreover, for every $m \geq 0$ we denote the full subcategory of $\Delta_s$ consisting of $[n]$ for all $0 \leq n \leq m$ by $\Delta_{s,\leq m}$.
\item
For an f-adic ring $A$, we denote the ring of power-bounded elements of $A$ by $A^{\circ}$.
\item
For a topological ring $A$, we denote the discrete ring whose underlying ring is the underlying ring of $A$ by $A_{\disc}$. 
\end{itemize}
\subsection*{Acknowledgements}
The author is grateful to Yoichi Mieda for his support during the studies of the author. In addition, the author is grateful to Ko Aoki for his answers in questions about higher algebra, and to Lucas Mann for his comments on this paper.
Finally, the author would like to the referee for several corrections.


\section{Review of analytic rings associated to complete affinoid pairs}
In this section, we recall some results of \cite{CM,AG,And21, Mann22}, which will be needed later.
Most of the propositions in this section are stated without proofs. For complete proofs, see \cite{CM,AG,And21, Mann22}.

\subsection{Analytic rings}
We begin with the definition of condensed sets, condensed abelian groups, condensed rings, etc.

\begin{definition}[{\cite[Definition 1.2]{CM}}]\label{defn:cond}
Let $ \ast_{\proet}$ denote the pro-\'etale site of the point $\ast$, i.e., the category of profinite sets with covers given by finite families of jointly surjective maps.
A \textit{condensed set} is a sheaf of sets on $ \ast_{\proet}$. Similarly, a condensed ring, group, etc.\  is a sheaf of rings, groups, etc.\  on $ \ast_{\proet}$ respectively.
We denote the category of condensed sets by $\Cond$.
\end{definition}

\begin{remark}
This definition has minor set-theoretic issues. For the correct definition, see \cite[Appendix to Lecture II]{CM} or \cite[\S 2.1]{Mann22}. 
Clausen-Scholze's solution to these issues is different from the one adopted in \cite{HTT}.
We often use the results of \cite{HTT}, and they are justified by the methods used in \cite[\S 2.9]{Mann22}.
In this paper, we will ignore this kind of problems.
\end{remark}

\begin{example}[{\cite[Proposition 2.15]{CM}}]
Let $X$ be a $T_1$ space (i.e., a topological space all of whose points are closed).
Then the presheaf $\underline{X}$ defined by sending a profinite set $S$ to the set of continuous maps from $S$ to $X$ becomes a condensed set.
Here, the $T_1$ condition is related to set-theoretical issues.
Similarly, for a Hausdorff topological abelian group or ring $A$, we can define a condensed abelian group or ring $\underline{A}$.
\end{example}

\begin{remark}
We can identify the category of condensed sets, rings, groups, etc.\  with the category of functors 
$$\{\mbox{extremally disconnected sets}\}^{\op} \to \{\mbox{sets, rings, groups, etc.}\}$$
sending finite disjoint unions to finite products, respectively (see \cite[Lecture II]{CM}).
However, it is not yet known that the $\infty$-category of sheaves of spaces on $ \ast_{\proet}$ is equivalent to the $\infty$-category of functors 
$$\{\mbox{extremally disconnected sets}\}^{\op} \to \Scal$$
sending finite disjoint unions to finite products, where $\Scal$ is the $\infty$-category of spaces. 
It is claimed in \cite[Warning 2.2.2]{BH19} that the former category is not hypercomplete but the latter category is hypercomplete.
In such a case, we will use the latter as condensed objects.
\end{remark}

For a condensed ring $\Acal$, we denote the category of condensed $\Acal$-modules by $\Mod_{\Acal}^{\cond}$, and the derived $\infty$-category of condensed $\Acal$-modules by $\Dcal(\Acal)$.

\begin{theorem}[{\cite[Theorem 1.10]{CM}}]\label{gro}
The category $\Mod_{\Acal}^{\cond}$ is an abelian category which satisfies Grothendieck'{}s axioms (AB3), (AB4), (AB5), (AB6), (AB3*), and (AB4*). Furthermore, it is generated by compact projective objects.
\end{theorem}

Note that, for an extremally disconnected set $S$,  the functor $\Mod_{\Acal}^{\cond} \to \Ab ;\; M \mapsto M(S)$ is exact. 
Let $\Acal[S]$ denote the sheafification of $T \mapsto \Acal(T)^{\oplus \underline{S}(T)}$. 
Then $\Hom_{\Acal}(\Acal[S], M)$ is isomorphic to $M(S)$ for every condensed $\Acal$-module $M$. 
Therefore, $\Acal[S]$ is a compact projective condensed $\Acal$-module for any extremally disconnected set $S$, and such modules form a family of compact projective generators (see \cite[Lecture II]{CM}).
Moreover $\Mod_{\Acal}^{\cond}$ has a tensor product and an internal Hom as usual (for details, see \cite[Proposition 2.1.11]{Mann22}).
We denote the latter by $\underline{\Hom}_{\Acal}(-,-)$. 
Note that $\underline{\Hom}_{\Acal}(M,N)(S)$ is equal to $\Hom_{\Acal}(M \otimes_{\Acal} \Acal[S], N)$ for each extremally disconnected set $S$.
In addition, for an object $M \in \Dcal(\Acal)$, we denote $R\Gamma(S,M) \in \Dcal(\Ab)$ by $M(S)$.

Next, we define (static) analytic rings.
\begin{definition}[{\cite[Definitions 7.1, 7.4]{CM}}]\label{def:analytic ring}
An \textit{uncompleted pre-analytic ring} $(\Acal,\Mcal)$ is a condensed ring $\Acal$ equipped with a functor 
$$\{\mathrm{extremally\ disconnected\ sets}\}\to \Mod^{\cond}_{\Acal};\; S\mapsto \Mcal[S]$$
that sends finite disjoint unions to finite products, 
and a natural transformation $\Phi_{(\Acal,\Mcal)}\colon \underline{S} \to \Mcal[S]$ of functors from the category of extremally disconnected sets to $\Cond.$
For an uncompleted pre-analytic ring $(\Acal, \Mcal)$, we call $\Acal$ the \textit{underlying condensed ring} of $(\Acal, \Mcal)$ and $\Mcal$ the \textit{functor of measures} of $(\Acal, \Mcal)$.
 An uncompleted pre-analytic ring $(\Acal, \Mcal)$ is called an \textit{uncompleted analytic ring} if, for every complex
$$C\colon \cdots \rightarrow C_i\rightarrow \cdots \rightarrow C_1\rightarrow C_0\rightarrow 0 $$
of $\Acal$-modules such that each $C_i$ is a direct sum of objects of the form $\Mcal[T]$ for various extremally disconnected sets $T$, 
the map $$R\intHom_{\Acal}(\Mcal[S], C) \to R\intHom_{\Acal}({\Acal}[S], C)$$ is an equivalence for all extremally disconnected sets $S$. 
An uncompleted analytic ring $(\Acal,\Mcal)$ is called an \textit{analytic ring} if the map $\Acal \to \Mcal[\ast]$ is an isomorphism.
\end{definition}

\begin{remark}
The above notation of analytic rings is used in \cite[Definition 12.1]{AG} and \cite[Definition 2.3.1]{Mann22}.
In \cite[Definition 7.1]{CM}, an analytic ring is denoted by $\Acal$ which is a condensed ring $\underline{\Acal}$ with a functor $$\{\mathrm{extremally\ disconnected\ sets}\}\to \Mod^{\cond}_{\underline{\Acal}};\; S\mapsto \Acal[S]$$ as in Definition \ref{def:analytic ring}.
We will use both notations, which will cause no confusion.
\end{remark}

\begin{remark}
Our terminology follows \cite{Mann22}. 
In \cite{CM}, uncompleted pre-analytic rings, uncompleted analytic rings, and analytic rings are called pre-analytic rings, analytic rings, and normalized analytic rings, respectively.
\end{remark}

\begin{definition}
Let $(\Acal, \Mcal)$ be an uncompleted analytic ring.
\begin{enumerate}
\item
We define $\Mod_{(\Acal,\Mcal)}^{\cond} \subset \Mod_{\Acal}^{\cond}$ to be the full subcategory of all $\Acal$-modules $M$ such that for all extremally disconnected sets $S$, the map
$$\Hom_{\Acal}(\Mcal[S], M) \to \Hom_{\Acal}(\Acal[S], M) \cong M(S)$$
is an isomorphism.
An object of $\Mod_{\Acal}^{\cond}$ is said to be \textit{$(\Acal,\Mcal)$-complete} if it lies in $\Mod_{(\Acal,\Mcal)}^{\cond}$.
\item
We define $\Dcal(\Acal,\Mcal) \subset \Dcal(\Acal)$ to be the full $\infty$-subcategory of all complexes of  $\Acal$-modules $M$ such that for all extremally disconnected sets $S$, the map
$$R\Hom_{\Acal}(\Mcal[S], M) \to R\Hom_{\Acal}(\Acal[S], M)$$ 
is an equivalence.
An object of $\Dcal(\Acal)$ is said to be \textit{$(\Acal,\Mcal)$-complete} if it lies in $\Dcal(\Acal,\Mcal)$.
\end{enumerate}
\end{definition}

\begin{proposition}[{\cite[Proposition 7.5]{CM}}, {\cite[Lemma 5.27]{And21}}]
Let $(\Acal,\Mcal)$ be an uncompleted analytic ring.
\begin{enumerate}
\item
\begin{itemize}
\item
The full subcategory 
$\Mod_{(\Acal,\Mcal)}^{\cond} \subset \Mod_{\Acal}^{\cond}$
is an abelian category stable under limits, colimits, and extensions.
The objects of the form $\Mcal[S]$, where $S$ is an extremally disconnected set, form a family of compact projective generators.
\item
The inclusion functor
$\Mod_{(\Acal,\Mcal)}^{\cond} \to \Mod_{\Acal}^{\cond} $
admits a left adjoint
$$\Mod_{\Acal}^{\cond} \to \Mod_{(\Acal,\Mcal)}^{\cond} ;\; M \mapsto M\otimes_{\Acal} (\Acal, \Mcal),$$
which is the unique colimit-preserving extension of $\Acal[S] \mapsto \Mcal[S]$.
There is a unique symmetric monoidal tensor product $-\otimes_{(\Acal,\Mcal)}-$ making the functor above symmetric monoidal.
\item
For $M,N \in \Mod_{(\Acal,\Mcal)}^{\cond}$, $\underline{\Hom}_{\Acal}(M,N)$ is also $(\Acal,\Mcal)$-complete, and $\underline{\Hom}_{\Acal}(-,-)$ becomes an internal Hom of $\Mod_{(\Acal,\Mcal)}^{\cond}$.
\end{itemize}
\item
\begin{itemize}
\item
The full $\infty$-subcategory $\Dcal(\Acal,\Mcal) \subset \Dcal(\Acal)$ is a stable subcategory stable under limits and colimits.
The image of the functor 
$\Dcal(\Mod_{(\Acal,\Mcal)}^{\cond}) \to \Dcal(\Acal)$
is included in $\Dcal(\Acal,\Mcal)$, and 
$\Dcal(\Mod_{(\Acal,\Mcal)}^{\cond}) \to \Dcal(\Acal,\Mcal)$
is an equivalence.
For $M \in \Dcal(\Acal,\Mcal)$, $$R\intHom_{\Acal}(\Mcal[S], M) \to R\intHom_{\Acal}(\Acal[S], M)$$ is also an equivalence.
An object $M \in \Dcal(\Acal)$ is $(\Acal,\Mcal)$-complete if and only if $H^i(M)$ is $(\Acal,\Mcal)$-complete for all $i\in \Zbb$.
\item
The inclusion functor
$\Dcal(\Acal,\Mcal) \to \Dcal(\Acal)$
admits a left adjoint
$$\Dcal(\Acal) \to \Dcal(\Acal,\Mcal) ;\; M \mapsto M\otimes_{\Acal}^{\Lbb} (\Acal, \Mcal),$$
which is a left derived functor of $M \mapsto M\otimes_{\Acal} (\Acal, \Mcal)$.
There is a unique symmetric monoidal tensor product $-\otimes_{(\Acal,\Mcal)}^{\Lbb}-$ making the functor above symmetric monoidal.
\item
For $M,N \in \Dcal(\Acal,\Mcal)$, $R\intHom_{\Acal}(M,N)$ is also $(\Acal,\Mcal)$-complete, and $R\intHom_{\Acal}(-,-)$ becomes an internal Hom of $\Dcal(\Acal,\Mcal)$.
\end{itemize}
\end{enumerate}
\end{proposition}

\begin{definition}[{\cite[Lecture VII, page 47]{CM}}]
Let $(\Acal,\Mcal)$ and $(\Bcal, \Ncal)$ be uncompleted analytic rings. 
A map of uncompleted analytic rings from $(\Acal,\Mcal)$ to $(\Bcal, \Ncal)$ is a map $\Acal \to \Bcal$ of underlying condensed rings such that for all extremally disconnected sets $S$, the $\Acal$-module $\Ncal[S]$ is $(\Acal,\Mcal)$-complete.
\end{definition}

\begin{remark}
For a map $(\Acal,\Mcal) \to (\Bcal, \Ncal)$ of uncompleted analytic rings, every object $M\in \Dcal(\Bcal, \Ncal)$ is $(\Acal,\Mcal)$-complete,
since $\Dcal(\Bcal, \Ncal)$ is generated by $\Ncal[S]$ for various extremally disconnected sets $S$ under shifts and sifted colimits and  $\Ncal[S]$ is $(\Acal,\Mcal)$-complete.
\end{remark}

\begin{proposition}[{\cite[Proposition 7.7]{CM}}]
Let $(\Acal,\Mcal)$ and $(\Bcal, \Ncal)$ be uncompleted analytic rings, and let $f \colon (\Acal,\Mcal) \to (\Bcal, \Ncal)$ be a map of uncompleted analytic rings.
\begin{enumerate}
\item
The composite functor
$$\Mod_{\Acal}^{\cond} \xrightarrow{- \otimes_{\Acal} \Bcal}\Mod_{\Bcal}^{\cond} \xrightarrow{-\otimes_{\Bcal} (\Bcal, \Ncal)} \Mod_{(\Bcal, \Ncal)}^{\cond}$$
factors over  $\Mod_{(\Acal, \Mcal)}^{\cond}$, via a functor denoted by
$$\Mod_{(\Acal, \Mcal)}^{\cond} \to \Mod_{(\Bcal, \Ncal)}^{\cond} ;\; M \to M\otimes_{(\Acal, \Mcal)}(\Bcal, \Ncal).$$
This is the left adjoint of the forgetful functor $\Mod_{(\Bcal,\Ncal)}^{\cond} \to \Mod_{(\Acal,\Mcal)}^{\cond}$.
\item
The composite functor
$$\Dcal(\Acal) \xrightarrow{- \otimes_{\Acal}^{\Lbb} \Bcal}\Dcal(\Bcal) \xrightarrow{-\otimes_{\Bcal}^{\Lbb} (\Bcal, \Ncal)} \Dcal(\Bcal, \Ncal)$$
factors over  $\Dcal(\Acal, \Mcal)$, via a functor denoted by 
$$\Dcal(\Acal, \Mcal) \to \Dcal(\Bcal, \Ncal) ;\; M \to M\otimes_{(\Acal, \Mcal)}^{\Lbb}(\Bcal, \Ncal).$$
This is the left adjoint functor of the forgetful functor $\Dcal(\Bcal,\Ncal) \to \Dcal(\Acal,\Mcal)$.
\end{enumerate}
\end{proposition}

Similarly, we can also make a definition of a (complete commutative) analytic animated ring $(\Acal, \Mcal)$, and similar propositions hold true for it.
In essence, however, we treat only (0-truncated) analytic rings in this paper, so we do not give the precise definition.
See \cite[Lectures XI, XII]{AG}, \cite[\S 2.3]{Mann22} for more details. 
We denote the $\infty$-category of (complete commutative) analytic animated rings by $\AnRing$.
Then the category of (complete) analytic rings becomes a full $\infty$-subcategory of $\AnRing$ in a natural way.

\begin{theorem}[{\cite[Proposition 12.12]{AG}}]
The $\infty$-category $\AnRing$ admits all small colimits. In particular, $\AnRing$ admits pushouts.
Moreover, sifted colimits in $\AnRing$ commute with the functor $(\Acal,\Mcal) \to \Mcal[S]$ to $\Dcal^{\leq 0}(\Cond(\Ab))$ for every extremally disconnected set $S$, where $\Cond(\Ab)$ is the category of condensed abelian groups.
\end{theorem}

\begin{proposition}[{\cite[Proposition 12.12]{AG}}]\label{p}
Let 
$$\xymatrix{
(\Acal, \Mcal_{\Acal}) \ar[r] \ar[d] &(\Bcal, \Mcal_{\Bcal}) \ar[d] \\
(\Ccal, \Mcal_{\Ccal}) \ar[r] &(\Ecal, \Mcal_{\Ecal}) 
}
$$
be a pushout diagram in $\AnRing$.
Then for any object $M\in \Dcal(\Ecal)$, $M$ is $(\Ecal, \Mcal_{\Ecal})$-complete if and only if $M$ is $(\Bcal, \Mcal_{\Bcal})$-complete and $(\Ccal, \Mcal_{\Ccal})$-complete.
\end{proposition}
 
The following lemma is useful for showing that a given diagram is a pushout diagram.
 
\begin{lemma}[{\cite[Proposition 12.12]{AG}}]\label{pushout}
Let 
$$\xymatrix{
(\Acal, \Mcal_{\Acal}) \ar[r] \ar[d] &(\Bcal, \Mcal_{\Bcal}) \ar[d] \\
(\Ccal, \Mcal_{\Ccal}) \ar[r] &(\Ecal, \Mcal_{\Ecal}) 
}
$$
be a commutative diagram in $\AnRing$.
We assume that the following are satisfied:
\begin{enumerate}
\item
The object $\Ecal$ is $(\Bcal, \Mcal_{\Bcal})$-complete, and the natural map $\Ccal \otimes_{(\Acal, \Mcal_{\Acal})}^{\Lbb} (\Bcal, \Mcal_{\Bcal}) \to \Ecal$ is an equivalence.
\item
For any object $M\in \Dcal(\Ecal)$, $M$ is $(\Ecal, \Mcal_{\Ecal})$-complete if and only if $M$ is $(\Bcal, \Mcal_{\Bcal})$-complete and $(\Ccal, \Mcal_{\Ccal})$-complete.
\end{enumerate}
Then the diagram above is a pushout.
\end{lemma}


\begin{definition}
A map $f \colon (\Acal, \Mcal_{\Acal}) \to (\Bcal, \Mcal_{\Bcal})$ in $\AnRing$ is \textit{steady} if for all maps $g \colon (\Acal, \Mcal_{\Acal}) \to (\Ccal, \Mcal_{\Ccal})$ in $\AnRing$ and all $M \in \Dcal(\Ccal, \Mcal_{\Ccal})$, 
the object $M \otimes_{(\Acal, \Mcal_{\Acal})}^{\Lbb} (\Bcal, \Mcal_{\Bcal})$, which is a priori an object of $\Dcal(\Ccal \otimes_{\Acal} \Bcal)$, is $(\Ccal, \Mcal_{\Ccal})$-complete.
\end{definition}

By definition, the class of steady maps is stable under base change and composition.

\begin{remark}
The original definition \cite[Definition 12.13]{AG} of steady maps is different from our definition.
It is claimed in \cite[page 83]{AG} that the original definition is equivalent to our definition, but this is not true.
However this causes no issues in the result of \cite{AG} because only our definition is used there.
\end{remark}

\begin{proposition}[{\cite[Proposition 12.14]{AG}}]\label{steady bc}
A map $f \colon (\Acal, \Mcal_{\Acal}) \to (\Bcal, \Mcal_{\Bcal})$ in $\AnRing$ is steady if and only if for all pushout diagrams
$$\xymatrix{
(\Acal, \Mcal_{\Acal}) \ar[r]^{f} \ar[d] &(\Bcal, \Mcal_{\Bcal}) \ar[d] \\
(\Ccal, \Mcal_{\Ccal}) \ar[r] &(\Ecal, \Mcal_{\Ecal}) 
}
$$
and all $M \in \Dcal(\Ccal, \Mcal_{\Ccal})$, the base change map
$$\left(M |_{\Acal}\right) \otimes_{(\Acal, \Mcal_{\Acal})}^{\Lbb} (\Bcal, \Mcal_{\Bcal}) \to ( M \otimes_{(\Ccal, \Mcal_{\Ccal})}^{\Lbb} (\Ecal, \Mcal_{\Ecal}) ) |_{\Bcal}$$
is an equivalence.
\end{proposition}

\subsection{Analytic rings associated to complete affinoid pairs}
We begin with the construction of (complete) analytic rings associated to complete affinoid pairs.
\begin{construction}[{\cite[Example 7.3, Theorem 8.1]{CM}}]\label{const}
Let $A$ be a finitely generated $\Zbb$-algebra. We define an uncompleted pre-analytic ring $A_{\bs}$ as follows:
\begin{itemize}
\item
The underlying ring is $\underline{A}$.
\item
For an extremally disconnected set $\displaystyle S= \varprojlim_{i} S_i$ (each $S_i$ is a finite set), we define $\displaystyle A_{\bs}[S]= \varprojlim_i \underline{A}[S_i]$.  
\item
The natural transform $\underline{S} \to A_{\bs}[S]$ is defined in the obvious way.
\end{itemize}
Then the uncompleted pre-analytic ring $A_{\bs}$ becomes an analytic ring.
\end{construction}

\begin{construction}[{\cite[Lemma 3.19, Definition 3.20]{And21}}]
Let $(A, A^+)$ be an affinoid pair such that $A$ is a discrete ring. We define an uncompleted pre-analytic ring $(A, A^+)_{\bs}$ as follows:
\begin{itemize}
\item
The underlying ring is $\underline{A}$.
\item
For an extremally disconnected set $S$, we define $\displaystyle (A, A^+)_{\bs}[S]= \varinjlim_{B \subset A^+} \underline{A} \otimes_{\underline{B}} B_{\bs}[S]$, where the colimit is taken over all finitely generated $\Zbb$-subalgebras $B$ of $A^+$.
\item
The natural transform $\underline{S} \to (A,A^+)_{\bs}[S]$ is defined in the obvious way.
\end{itemize}
Then the uncompleted pre-analytic ring $(A, A^+)_{\bs}$ becomes an analytic ring. If $A^+$ is equal to $A$, then we denote $(A, A^+)_{\bs}$ by $A_{\bs}$. This notation is compatible with that of Construction \ref{const}.
\end{construction}

\begin{construction}[{\cite[Lemma 3.25, Theorem 3.28]{And21}}]\label{construction}
Let $(A, A^+)$ be a complete affinoid pair. We define an uncompleted pre-analytic ring $(A, A^+)_{\bs}$ as follows:
\begin{itemize}
\item
The underlying ring is $\underline{A}$.
\item
For an extremally disconnected set $S$, we define $(A, A^+)_{\bs}[S]= \underline{A} \otimes_{(A_{\disc}, A^+_{\disc})_{\bs}} (A_{\disc}, A^+_{\disc})_{\bs}[S]$.
\item
The natural transform $\underline{S} \to (A,A^+)_{\bs}[S]$ is defined in the obvious way.
\end{itemize}
Then the uncompleted pre-analytic ring $(A, A^+)_{\bs}$ becomes an analytic ring. If $A^+$ is equal to $A$, then we denote $(A, A^+)_{\bs}$ by $A_{\bs}$.
\end{construction}

We recall some results about such analytic rings.

\begin{proposition}[{\cite[Proposition 3.22]{And21}}]\label{integral}
Let $f \colon A \to B$ be an integral map of discrete rings. 
Then $M \in \Dcal(\underline{B})$ is $B_{\bs}$-complete if and only if $M$ is $A_{\bs}$-complete.
\end{proposition}

\begin{proposition}[{\cite[Proposition 3.32]{And21}}]\label{prop:int}
Let $A$ be an f-adic ring, and $f_I=\{f_i \mid i \in I \}$ be a collection of elements of $A^{\circ}$.
Let $A^+$ be the minimal ring of integral elements containing $f_I$.
Then $M \in \Dcal(\underline{A})$ is $(A, A^+)_{\bs}$-complete if and only if $M$ is $\Zbb[f_i]_{\bs}$-complete for all $i \in I$, where we regard $\Zbb[f_i]$ as a discrete ring.
\end{proposition}

\begin{remark}
It follows from Proposition \ref{integral} and Proposition \ref{prop:int} that for a map $(A,A^+)\to (B,B^+)$ of complete affinoid pairs, $(B,B^+)_{\bs}[S]$ is $(A,A^+)_{\bs}$-complete for any extremally disconnected set $S$.
In particular, we get a map $(A,A^+)_{\bs}\to (B,B^+)_{\bs}$ of analytic rings.
\end{remark}

\begin{proposition}[{\cite[Proposition 3.34]{And21}}]
Let $\cAff$ denote the category of complete affinoid pairs.
Then the following functor is fully faithful:
$$ \cAff \to \AnRing ;\; (A, A^+) \mapsto (A, A^+)_{\bs}.$$
Moreover the image of this functor is contained in the full $\infty$-subcategory consisting of ($0$-truncated) analytic rings.
\end{proposition}


\begin{proposition}[{\cite[Proposition 3.33]{And21}}]\label{prop:dp} 
Let $A$ be a discrete ring, and $B,C$ be discrete $A$-algebras such that $B \otimes_{A} C \simeq B \otimes_{A}^{\Lbb} C$.
Then we have the following equivalence:
$$B_{\bs} \otimes_{A_{\bs}}^{\Lbb} C_{\bs} \simeq (B \otimes_{A} C)_{\bs}.$$
\end{proposition}

\begin{lemma}[{\cite[Lemma 3.8]{And21}}]\label{poly steady}
The natural map $\Zbb_{\bs} \to \Zbb[T]_{\bs}$ in $\AnRing$ is steady.
\end{lemma}

\begin{lemma}[{\cite[Lemma 3.14]{And21}}]\label{tensor}
Let $M$ be a prodiscrete topological abelian group. Let $M\langle T_1,\ldots,T_n \rangle$ denote the topological abelian group of convergent formal power series with $n$ variables.
Then we have the following equivalence:
$$M \otimes_{\Zbb_{\bs}}^{\Lbb} \Zbb[T_1,\ldots,T_n]_{\bs} \simeq \underline{M\langle T_1,\ldots,T_n \rangle}.$$
\end{lemma}

\begin{lemma}[{\cite[Lemma 4.7]{And21}}]\label{poly bs}
Let $(A,A^+)$ be a complete affinoid pair. Then we have the following equivalence:
$$(A\langle T\rangle,A^+\langle T\rangle)_{\bs} \simeq (A, A^+)_{\bs} \otimes_{\Zbb_{\bs}}^{\Lbb} \Zbb[T]_{\bs}.$$
\end{lemma}

\begin{theorem}[{\cite[Theorem 4.1]{And21}}]\label{andes}
Let $X$ be an analytic affinoid adic space and let $U$ denote an arbitrary open affinoid subspace of $X$.
Then the functor $U \mapsto \Dcal((\Ocal_X(U), \Ocal_X^+(U))_{\bs})$ defines a sheaf of $\infty$-categories on $X$.
\end{theorem}

Finally, we compute $(A,A^+)_{\bs}[S]$ for an extremally disconnected set $S$. 

\begin{lemma}[{\cite[Corollary 5.5]{CM}}]\label{lem}
There exists a set $J$ and an isomorphism of condensed abelian groups
$$\Zbb_{\bs}[S] \cong \prod_{J} \underline{\Zbb}.$$
\end{lemma}

\begin{definition}[{\cite[Definition 3.4]{And21}}]
Let $A$ be an f-adic ring, and $B$ be a finitely generated $\Zbb$-subalgebra of $A^{\circ}$.
Then a $B$-submodule $M$ of $A$ is said to be \textit{quasi-finitely generated} over $B$ if $M$ is closed in $A$ and for some pair of definition $(A_0, I)$ such that $B \subset A_0$, the image of $M \to A/(I^nA_0)$ is a finitely generated $B$-module for any positive integer $n$.
\end{definition}

\begin{remark}
If $M$ is quasi-finitely generated over $B$, then for any pair of definition $(A_0, I)$ such that $B \subset A_0$, the image of $M \to A/(I^nA_0)$ is a finitely generated $B$-module for any positive integer $n$.
\end{remark}
\begin{remark}\label{rem:dis}
When $A$ is discrete, a $B$-submodule $M$ of $A$ is quasi-finitely generated over $B$ if and only if $M$ is finitely generated over $B$.
\end{remark}

\begin{theorem}[{\cite[Theorem 3.27]{And21}}]\label{dir prod}
Let $(A,A^+)$ be a complete affinoid pair, and let $J$ be a set such that $\displaystyle\Zbb_{\bs}[S] \cong \prod_{J} \Zbb$. Then we have an isomorphism
$$(A, A^+)_{\bs}[S] \cong \varinjlim_{B \subset A^+, M} \prod_J \underline{M},$$
where the colimit is taken over all the finitely generated $\Zbb$-subalgebras $B \subset A^+$ and all the quasi-finitely generated $B$-submodules $M$ of $A$.
\end{theorem}

\begin{remark}\label{rem:dirprod}
We assume that $A$ is discrete.
Then by Remark \ref{rem:dis}, we get an isomorphism
$$(A, A^+)_{\bs}[S] \cong \varinjlim_{B \subset A^+} \prod_J \underline{M}.$$
\end{remark}

\begin{corollary}\label{cpt}
Let $(A,A^+)$ be a complete affinoid pair, and let $J$ be a set. Then $\displaystyle\varinjlim_{B \subset A^+, M} \prod_J \underline{M}$ is a compact projective object of $\Mod_{(A,A^+)_{\bs}}^{\cond}$.
\end{corollary}

\section{Faithfully flat descent for discrete rings}
In what follows, we prove a version of fppf descent for discrete rings in condensed mathematics. 
In this section, all rings are discrete (static) topological rings unless otherwise stated.

\begin{definition}
Let $\Acal$ be a condensed ring. Then a \textit{finite projective $\Acal$-module} is a condensed $\Acal$-module which is a direct summand of the finite free $\Acal$-module $\Acal^n$ for some $n$.
\end{definition}

\begin{definition}[{\cite[Lemma 5.7]{And21}}]
Let $A$ be a ring and endow every $A$-module with the discrete topology.
Then the functor $\Mod_{A} \to \Mod_{\underline{A}}^{\cond} ;\; M \mapsto \underline{M}$ extends to a fully faithful exact functor $\dCond_{A} \colon \Dcal(A) \to \Dcal(A_{\bs})$, which is called the \textit{discrete condensification functor}.
For an object $M\in \Dcal(A)$, $\dCond_A(M)$ is given by $\underline{N^{\bullet}}$, where $N^{\bullet}$ is any complex representing $M$.
\end{definition}

\begin{remark}\label{cond}
From this, we find that for a perfect $A$-module $M$ (i.e., $M[0]$ is a perfect complex), $\underline{M}$ is quasi-isomorphic to a bounded complex of finite projective $\underline{A}$-modules.
\end{remark}

\begin{lemma}\label{perfect}
Let $A$ be a ring and $M$ be a pseudo-coherent complex of $A$-modules. 
Let $a,b\in \Zbb$ such that $a \leq b$.
If for every maximal ideal $\mfrak \subset A$, we have $H^i(M \otimes^{\Lbb}_{A} A/\mfrak)=0$ for all $i\notin [a,b]$, then $M$ is a perfect complex with Tor-amplitude in $[a,b]$.
\end{lemma}
\begin{proof}
It easily follows from \cite[Corollary 8.3.6.5]{FGA} or \cite[\href{https://stacks.math.columbia.edu/tag/068V}{Tag 068V}]{stacks-project}.
\end{proof}

By using the above lemma, we get the following decomposition (cf.\ \cite[\href{https://stacks.math.columbia.edu/tag/068Y}{Tag 068Y}]{stacks-project}).

\begin{lemma}\label{disc rp} 
Let $A$ be a finitely generated $\Zbb$-algebra, and $B$ be a finitely presented flat $A$-algebra. Then there exists a decomposition 
$$A \to A[T_1,\ldots,T_n] \to B$$
such that $B$ is a perfect $A[T_1,\ldots,T_n]$-module.
\end{lemma}
\begin{proof}
We take a surjection $A[T_1,\ldots,T_n] \to B$, and we will show that $B$ is a perfect $A[T_1,\ldots,T_n]$-module.
We take any maximal ideal $\mfrak$ of $A[T_1,\ldots,T_n]$.
Since $A$ is a Jacobson ring, the prime ideal $\nfrak = \mfrak \cap A$ of $A$ is a maximal ideal.
Therefore, the global dimension of $A/\nfrak[T_1,\ldots,T_n]$ is equal to $n$.
We have the following equivalence:
\begin{align*}
&B \otimes_{A[T_1,\ldots,T_n]}^{\Lbb} A[T_1,\ldots,T_n]/\mfrak \\
\simeq{} & (B \otimes_{A[T_1,\ldots,T_n]}^{\Lbb} A/\nfrak[T_1,\ldots,T_n]) \otimes_{A/\nfrak[T_1,\ldots,T_n]}^{\Lbb} A[T_1,\ldots,T_n]/\mfrak \\
\simeq{} & (B \otimes_{A[ T_1,\ldots,T_n]}^{\Lbb} (A[T_1,\ldots,T_n]\otimes_A^{\Lbb} A/\nfrak)) \otimes_{A/\nfrak[T_1,\ldots,T_n]}^{\Lbb} A[T_1,\ldots,T_n]/\mfrak \\
\simeq{} & (B \otimes_A^{\Lbb} A/\nfrak) \otimes_{A/\nfrak[T_1,\ldots,T_n]}^{\Lbb} A[T_1,\ldots,T_n]/\mfrak \\
\simeq{} &B/{\nfrak B} \otimes_{A/\nfrak[T_1,\ldots,T_n]}^{\Lbb} A[T_1,\ldots,T_n]/\mfrak,
\end{align*}
where the flatness of $B$ over $A$ is used in the fourth equivalence.
Therefore, we get $H^i(B \otimes_{A[T_1,\ldots,T_n]}^{\Lbb} A[T_1,\ldots,T_n]/\mfrak)=0$ for all $i \notin [-n, 0]$.
Since $A[T_1,\ldots,T_n]$ is noetherian, $B$ is a pseudo-coherent $A[T_1,\ldots,T_n]$-module, 
so we get the claim by Lemma \ref{perfect}.
\end{proof}

\begin{proposition}\label{steady2}
Let $A$ be a ring, and $B$ be a finitely presented flat $A$-algebra. Then $A_{\bs} \to B_{\bs}$ is steady. 
\end{proposition}
\begin{proof}
By a limit argument, we can take a finitely generated $\Zbb$-subalgebra $A_0$ of $A$ and a finitely presented flat $A_0$-algebra $B_0$ such that $A \otimes_{A_0} B_0 \cong B$.
By Proposition \ref{prop:dp}, we have an equivalence $A_{\bs} \otimes_{(A_0)_{\bs}}^{\Lbb} (B_0)_{\bs} \simeq B_{\bs}$. Since the class of steady maps is stable under base change, we may assume that $A$ is a finitely generated $\Zbb$-algebra.
Since the class of steady maps is stable under composition, we may assume that $B$ is equal to $A[T]$ or $B$ is a perfect $A$-module by Lemma \ref{disc rp}.
If $B$ is equal to $A[T]$, then $B_{\bs}$ is equivalent to $A_{\bs} \otimes_{\Zbb_{\bs}}^{\Lbb} \Zbb[T]_{\bs}$.
Therefore, $A_{\bs} \to B_{\bs}$ is steady, since $\Zbb_{\bs} \to \Zbb[T]_{\bs}$ is steady by Lemma \ref{poly steady}.
Next, we show the claim when $B$ is a perfect $A$-module.
Since $\underline{B}$ is quasi-isomorphic to a bounded complex of finite projective $\underline{A}$-modules, 
for an object $M \in \Dcal(A_{\bs})$, $M \otimes_{\underline{A}}^{\Lbb} \underline{B}$ is $A_{\bs}$-complete and also $B_{\bs}$-complete by Proposition \ref{integral}. 
This implies that $M \otimes_{A_{\bs}}^{\Lbb} B_{\bs}$ is equivalent to $M\otimes_{\underline{A}}^{\Lbb} \underline{B}$. 
Therefore, for all maps $A_{\bs}\to (\Ccal, \Mcal)$ in $\AnRing$ and all $M \in \Dcal(\Ccal, \Mcal)$, $M \otimes_{A_{\bs}}^{\Lbb} B_{\bs} \simeq M\otimes_{\underline{A}}^{\Lbb} \underline{B}$ is $(\Ccal, \Mcal)$-complete since $\underline{B}$ is quasi-isomorphic to a bounded complex of finite projective $\underline{A}$-modules. 
From the above, we get that $A_{\bs} \to B_{\bs}$ is steady. 
\end{proof}

\begin{remark}
By the proof of this proposition, we can also show that $A_{\bs} \to B_{\bs}$ is steady if $A \to B$ is a perfect ring map, i.e., there exists a decomposition as in Lemma \ref{disc rp}.
More generally, it is proved in \cite[Proposition 2.9.7]{Mann22} that for every ring map $A\to B$ of discrete rings,  $A_{\bs} \to B_{\bs}$ is steady.
\end{remark}

\begin{lemma}\label{eq}
Let $J$ be a set and $A$ be a finitely generated $\Zbb$-algebra. We have the following equivalence:
\begin{enumerate}
\item
$\displaystyle R\intHom_{\underline{\Zbb}}(\bigoplus_{J} \underline{\Zbb}, \underline{\Zbb}) \simeq \prod_{J} \underline{\Zbb}$.
\item
$\displaystyle R\intHom_{\underline{\Zbb}}(\prod_{J} \underline{\Zbb}, \underline{\Zbb}) \simeq \bigoplus_{J} \underline{\Zbb}$.
\item
$\displaystyle (\prod_{J} \underline{\Zbb}) \otimes_{\Zbb_{\bs}}^{\Lbb} A_{\bs} \simeq \prod_{J} \underline{A}$.
\end{enumerate}
\end{lemma}
\begin{proof}
The first equivalence follows from Theorem \ref{gro}.
The second equivalence follows from \cite[Theorem 4.3]{CM} and its proof.
For the third equivalence, we can take an extremally disconnected set $S$ and a set $I$ containing $J$ as a subset such that $\displaystyle\Zbb_{\bs}[S] \cong \prod_{I} \underline{\Zbb}$. 
Then the claim follows from the equivalence $\Zbb_{\bs}[S] \otimes_{\Zbb_{\bs}}^{\Lbb} A_{\bs} \simeq A_{\bs}[S]$ and 
$\displaystyle A_{\bs}[S] \cong \prod_{I} \underline{A}$ by Theorem \ref{dir prod}.
\end{proof}

The following lemma is key in this paper.

\begin{lemma}\label{poly}
We have an equivalence of functors from $\Dcal(\Zbb_{\bs})$ to $\Dcal(\Zbb[T_1,\ldots,T_n]_{\bs})$
$$-\otimes_{\Zbb_{\bs}}^{\Lbb} \Zbb[T_1,\ldots,T_n]_{\bs} \overset{\sim}{\lra} R\intHom_{\underline{\Zbb}}(R\intHom_{\underline{\Zbb}}(\underline{\Zbb[T_1,\ldots,T_n]}, \underline{\Zbb}), -).$$
\end{lemma}
\begin{proof}
We begin with the outline of the proof.
We have a natural map of functors from $\Dcal(\Zbb_{\bs})$ to $\Dcal(\Zbb_{\bs})$
$$\id=R\intHom_{\underline{\Zbb}}(R\intHom_{\underline{\Zbb}}(\underline{\Zbb},\underline{\Zbb}), -) \to R\intHom_{\underline{\Zbb}}(R\intHom_{\underline{\Zbb}}(\underline{\Zbb[T_1,\ldots,T_n]}, \underline{\Zbb}), -)$$
which is induced by the map $\underline{\Zbb} \to \underline{\Zbb[T_1,\ldots,T_n]}$.
First, we will show that  for any object $M \in \Dcal(\Zbb_{\bs})$, $R\intHom_{\underline{\Zbb}}(R\intHom_{\underline{\Zbb}}(\underline{\Zbb[T_1,\ldots,T_n]}, \underline{\Zbb}), M)$ is $\Zbb[T_1,\ldots,T_n]_{\bs}$-complete. 
Then we get a map of functors from $\Dcal(\Zbb_{\bs})$ to $\Dcal(\Zbb[T_1, \ldots, T_n]_{\bs})$
$$-\otimes_{\Zbb_{\bs}}^{\Lbb} \Zbb[T_1,\ldots,T_n]_{\bs} \to R\intHom_{\underline{\Zbb}}(R\intHom_{\underline{\Zbb}}(\underline{\Zbb[T_1,\ldots,T_n]}, \underline{\Zbb}), -),$$
and we will show that this is an equivalence.

We take an object $M \in \Dcal(\Zbb_{\bs})$, and we will show that 
$$R\intHom_{\underline{\Zbb}}(R\intHom_{\underline{\Zbb}}(\underline{\Zbb[T_1,\ldots,T_n]}, \underline{\Zbb}), M)$$
is $\Zbb[T_1,\ldots,T_n]_{\bs}$-complete. 
The object $R\intHom_{\underline{\Zbb}}(\underline{\Zbb[T_1,\ldots,T_n]}, \underline{\Zbb}) \in \Dcal({\Zbb_{\bs}})$ is equivalent to $\prod \underline{\Zbb}$ by Lemma \ref{eq}, so it is compact as an object of $\Dcal({\Zbb_{\bs}})$.
Therefore, the functor 
\begin{align}
R\intHom_{\underline{\Zbb}}(R\intHom_{\underline{\Zbb}}(\underline{\Zbb[T_1,\ldots,T_n]}, \underline{\Zbb}), -) \colon \Dcal(\Zbb_{\bs}) \to \Dcal(\underline{\Zbb[T_1,\ldots,T_n]}) \label{fct1}
\end{align}
commutes with small colimits.
Since the objects of the form $\Zbb_{\bs}[S]$ for an extremally disconnected set $S$ form a family of compact generators of $\Dcal(\Zbb_{\bs})$ and $\Dcal(\Zbb[T_1,\ldots,T_n]_{\bs}) \subset \Dcal(\underline{\Zbb[T_1,\ldots,T_n]})$ is stable under small colimits, 
we may assume that $M$ is equal to $\Zbb_{\bs}[S]$ for an extremally disconnected set $S$. 
By Lemma \ref{lem}, we have an isomorphism $\displaystyle\Zbb_{\bs}[S] \cong \prod_{J} \underline{\Zbb}$.
Since the functor (\ref{fct1})
commutes with small limits and $\Dcal(\Zbb[T_1,\ldots,T_n]_{\bs}) \subset \Dcal(\underline{\Zbb[T_1,\ldots,T_n]})$ is stable under small limits, 
we may assume that $M$ is equal to $\underline{\Zbb}$. By Lemma \ref{eq} (1), (2), the natural map 
$$\underline{\Zbb[T_1,\ldots,T_n]} \to R\intHom_{\underline{\Zbb}}(R\intHom_{\underline{\Zbb}}(\underline{\Zbb[T_1,\ldots,T_n]}, \underline{\Zbb}), \underline{\Zbb})$$
is an equivalence. In particular, 
$$R\intHom_{\underline{\Zbb}}(R\intHom_{\underline{\Zbb}}(\underline{\Zbb[T_1,\ldots,T_n]}, \underline{\Zbb}), \underline{\Zbb})$$ 
is $\Zbb[T_1,\ldots,T_n]_{\bs}$-complete. 

Finally, we prove that the map of functors from $\Dcal(\Zbb_{\bs})$ to $\Dcal(\Zbb[T_1, \ldots, T_n]_{\bs})$
$$-\otimes_{\Zbb_{\bs}}^{\Lbb} \Zbb[T_1,\ldots,T_n]_{\bs} \to R\intHom_{\underline{\Zbb}}(R\intHom_{\underline{\Zbb}}(\underline{\Zbb[T_1,\ldots,T_n]}, \underline{\Zbb}), -)$$
is an equivalence.
Since both functors commute with small colimits, it is enough to show that 
$$\prod_{J} \underline{\Zbb} \otimes_{\Zbb_{\bs}}^{\Lbb} \Zbb[T_1,\ldots,T_n]_{\bs} \to R\intHom_{\underline{\Zbb}}(R\intHom_{\underline{\Zbb}}(\underline{\Zbb[T_1,\ldots,T_n]}, \underline{\Zbb}), \prod_{J} \underline{\Zbb})$$
is an equivalence for any set $J$. We may also assume that $J$ is a singleton by Lemma \ref{eq} (3), and in this case the claim is already shown.
\end{proof}

\begin{theorem}\label{key1}
Let $A$ be a finitely generated $\Zbb$-algebra and $B$ be a finitely presented flat $A$-algebra.
Then we have an equivalence of functors from $\Dcal(A_{\bs})$ to $\Dcal(B_{\bs})$
$$-\otimes_{A_{\bs}}^{\Lbb} B_{\bs} \overset{\sim}{\lra} R\intHom_{\underline{A}}(R\intHom_{\underline{A}}(\underline{B}, \underline{A}), -).$$
\end{theorem}

\begin{proof}
In this proof, we denote $A[T_1, \ldots, T_n]$ by $A_n$.
We assume that $B=A_n$ or $B$ is a perfect $A$-module. 
By the same argument as in the proof of Lemma \ref{poly}, it is enough to show the following: 
\begin{enumerate}
\item
$R\intHom_{\underline{A}}(\underline{B}, \underline{A})$ is compact as an object of $\Dcal(A_{\bs})$.
\item
The map $\underline{B} \to R\intHom_{\underline{A}}(R\intHom_{\underline{A}}(\underline{B}, \underline{A}), \underline{A})$ is an equivalence.
\item
The functor $- \otimes_{A_{\bs}}^{\Lbb} B_{\bs}$ commutes with small limits.
\end{enumerate}
First, we check them for $B=A_n$.
We have an equivalence $R\intHom_{\underline{A}}(\underline{A_n}, \underline{A}) \simeq \prod \underline{A}$, and this is compact by Corollary \ref{cpt}, which proves (1).
Next, we prove (2).
By Lemma \ref{eq} (3), we have an equivalence $R\intHom_{\underline{A}}(\underline{A_n}, \underline{A}) \simeq R\intHom_{\underline{\Zbb}}(\underline{\Zbb[T_1, \ldots, T_n]}, \underline{\Zbb}) \otimes_{\Zbb_{\bs}}^{\Lbb} A_{\bs}$.
Therefore, we get 
\begin{align*}
&R\intHom_{\underline{A}}(R\intHom_{\underline{A}}(\underline{A_n}, \underline{A}), \underline{A}) \\
 \simeq{} & R\intHom_{\underline{A}}(R\intHom_{\underline{\Zbb}}(\underline{\Zbb[T_1, \ldots, T_n]}, \underline{\Zbb}) \otimes_{\Zbb_{\bs}}^{\Lbb} A_{\bs}, \underline{A}) \\
 \simeq{} & R\intHom_{\underline{\Zbb}}(R\intHom_{\underline{\Zbb}}(\underline{\Zbb[T_1, \ldots, T_n]}, \underline{\Zbb}), \underline{A}) \\
 \simeq{} & \underline{A} \otimes_{\Zbb_{\bs}}^{\Lbb} \Zbb[T_1,\ldots,T_n]_{\bs} \\
 \simeq{} & \underline{A_n},
\end{align*}
where the third equivalence follows from Lemma \ref{poly} and the fourth equivalence follows from Lemma \ref{tensor}.
Finally we prove (3).
We have the equivalence 
$$(A_n)_{\bs} \simeq A_{\bs} \otimes_{\Zbb_{\bs}}^{\Lbb} \Zbb[T_1, \ldots,T_n]_{\bs}$$ 
by Lemma \ref{poly bs}, where we note that $A_n \cong A\langle T_1, \ldots,T_n \rangle$ since $A$ is discrete.
By Proposition \ref{steady2}, $\Zbb_{\bs} \to \Zbb[T_1, \ldots,T_n]_{\bs}$ is steady.
According to Proposition \ref{steady bc}, it suffices to show that the functor 
$$-\otimes_{\Zbb_{\bs}}^{\Lbb} \Zbb[T_1, \ldots,T_n]_{\bs} \colon \Dcal(\Zbb_{\bs}) \to \Dcal(\Zbb[T_1, \ldots,T_n]_{\bs})$$
commutes with small limits, which follows from Lemma \ref{poly}.

Next, we check the conditions (1), (2), and (3) when $B$ is a perfect $A$-module.
The complex $\underline{B}$ is quasi-isomorphic to a bounded complex of finite projective $\underline{A}$-modules, and the functor $- \otimes_{A_{\bs}}^{\Lbb} B_{\bs}$ is equivalent to $- \otimes_{\underline{A}}^{\Lbb} \underline{B}$ by the proof of Proposition \ref{steady2}, so the conditions are independent of the ring structure of $B$ (i.e., they depend only on the $A$-module structure on $B$).
Therefore, we may assume that $B$ is a finitely generated projective $A$-module, and then the claim is clear. 

Finally, we treat a general case. We take a decomposition $A \to A_n \to B$ as Lemma \ref{disc rp}. By the discussion above, we have an equivalence of functors from $\Dcal(A_{\bs})$ to $\Dcal(B_{\bs})$ 
\begin{align*}
&-\otimes_{A_{\bs}}^{\Lbb} B_{\bs} \\
\simeq{} &R\intHom_{\underline{A_n}}(R\intHom_{A_n}(\underline{B}, \underline{A_n}), R\intHom_{\underline{A}}(R\intHom_{\underline{A}}(\underline{A_n}, \underline{A}), -))\\
\simeq{} & R\intHom_{\underline{A}}(R\intHom_{A_n}(\underline{B}, \underline{A_n}) \otimes_{\underline{A_n}}^{\Lbb} R\intHom_{\underline{A}}(\underline{A_n}, \underline{A}), -).
\end{align*}
Therefore, it is enough to show that the natural map 
$$R\intHom_{A_n}(\underline{B}, \underline{A_n}) \otimes_{\underline{A_n}}^{\Lbb} R\intHom_{\underline{A}}(\underline{A_n}, \underline{A}) \\
\to R\intHom_{\underline{A}}(\underline{B}, \underline{A})$$
is an equivalence. This follows from the fact that $\underline{B}$ is quasi-isomorphic to a bounded complex of finite projective $\underline{A_n}$-modules.
\end{proof}

\begin{corollary}\label{limit}
Let $A\to B$ be a finitely presented flat map of rings.
Then there exists an object $N_{B/A} \in \Dcal(\underline{B})$ satisfying the following conditions:
\begin{itemize}
\item
The object $N_{B/A}$ is $A_{\bs}$-complete and compact as an object of $\Dcal(A_{\bs})$.
\item
We have an equivalence of functors from $\Dcal(A_{\bs})$ to $\Dcal(B_{\bs})$
\begin{align}
-\otimes_{A_{\bs}}^{\Lbb} B_{\bs} \overset{\sim}{\lra} R\intHom_{\underline{A}}(N_{B/A}, -).
\label{eq3}
\end{align}
\end{itemize}
In particular, the functor
$$- \otimes_{A_{\bs}}^{\Lbb} B_{\bs} \colon \Dcal(A_{\bs}) \to \Dcal(B_{\bs})$$
commutes with small limits.
\end{corollary}

\begin{proof}
By a limit argument, we can take a finitely generated $\Zbb$-subalgebra $A_0$ of $A$ and a finitely presented flat $A_0$-algebra $B_0$ such that $A \otimes_{A_0} B_0 \cong B$.
We define $N_{B/A}=R\intHom_{\underline{A_0}}(\underline{B_0}, \underline{A_0})\otimes_{(A_0)_\bs}^{\Lbb} A_{\bs}$ which is an object of $\Dcal(\underline{B})$.
By Theorem \ref{key1} we have the following equivalence for every object $M\in \Dcal(A_{\bs})$:
$$R\intHom_{\underline{A}}(N_{B/A}, M) \simeq R\intHom_{\underline{A_0}}(R\intHom_{\underline{A_0}}(\underline{B_0}, \underline{A_0}), M) \simeq M\otimes_{(A_0)_{\bs}}^{\Lbb} (B_0)_{\bs}.$$
Since $R\intHom_{\underline{A}}(N_{B/A}, M)$ is $A_{\bs}$-complete and $M\otimes_{(A_0)_{\bs}}^{\Lbb} (B_0)_{\bs}$ is $(B_0)_{\bs}$-complete,
we find that $R\intHom_{\underline{A}}(N_{B/A}, M)$ is also $B_{\bs}$-complete, where we use Proposition \ref{p} and Proposition \ref{prop:dp}.
We have a natural map $N_{B/A} \to \underline{A}$ induced by the natural map $R\intHom_{\underline{A_0}}(\underline{B_0}, \underline{A_0})\to \underline{A_0}$.
Therefore, we get a natural map of functors from $\Dcal(A_{\bs})$ to $\Dcal(B_{\bs})$
$$-\otimes_{A_{\bs}} B_{\bs} \to R\intHom_{\underline{A}}(N_{B/A}, -)$$
induced by the natural map $N_{B/A} \to \underline{A}$, and it is an equivalence by Proposition \ref{steady2} and Proposition \ref{steady bc}.
\end{proof}

\begin{remark}\label{functoriality}
The equivalence (\ref{eq3}) is functorial in the following sense:
Let $\Alg_{A}^{\mathrm{fpf}}$ be the category of finitely presented flat $A$-algebras endowed with the coCartesian monoidal structure (i.e., the monoidal structure given by tensor products over $A$), 
and let $\End_{\Dcal(A_{\bs})}(\Dcal(A_{\bs}))$ be the $\infty$-category of $\Dcal(A_{\bs})$-enriched endofunctors (for the definition, see \cite[Definition A.4.4]{Mann22}).
Then there exists a monoidal functor $$N \colon \Alg_{A}^{\mathrm{fpf}} \to \Dcal(A_{\bs})^{\op};\; B \mapsto N_{B/A},$$ and an equivalence of functors from $\Alg_{A}^{\mathrm{fpf}} \to \End_{\Dcal(A_{\bs})}(\Dcal(A_{\bs}))$
$$B \mapsto -\otimes_{A_{\bs}}^{\Lbb} B_{\bs},$$
and
$$B \mapsto R\intHom_{\underline{A}}(N_{B/A}, -).$$

To see this, first observe that we have monoidal functors $$F\colon\Alg_{A}^{\mathrm{fpf}} \to \End_{\Dcal(A_{\bs})}(\Dcal(A_{\bs})) ;\; B \mapsto -\otimes_{A_{\bs}}^{\Lbb} B_{\bs}$$ and
$$G\colon\Dcal(A_{\bs})^{\op} \to \End_{\Dcal(A_{\bs})}(\Dcal(A_{\bs}));\; M\mapsto R\intHom_{\underline{A}}(M, -),$$
where the monoidal structure on $\End_{\Dcal(A_{\bs})}(\Dcal(A_{\bs}))$ is defined by compositions and the monoidal structure on $F$ is defined by Proposition \ref{steady bc} and Proposition \ref{steady2}.
By \cite[Corollary A.4.9]{Mann22} $G$ is fully faithful, and the image of $F$ is included in the essential image of $G$ by Corollary \ref{limit}, so we get the desired functoriality result.
\end{remark}

From the corollary above, we find that it is enough to examine properties of $N_{B/A}$ to prove faithfully flat descent.
To examine them, we will use a property of maps of rings introduced by Mathew in \cite{Mat16}.

\begin{definition}[{\cite[Definition 3.18, Proposition 3.20]{Mat16}}]\label{descendable def}
Let $\Ccal$ be a presentable symmetric monoidal stable $\infty$-category where the $\otimes$-product commutes with small colimits in each variable, which is called \textit{a stable homotopy theory} in \cite{Mat16}.
A map $A \to B$ of $\Ebb_{\infty}$-algebras in $\Ccal$ is \textit{descendable} if the pro-object $\{\Tot_{n}(B^{\bullet/A})\}_n$ of $\Dcal(A)$ is a constant pro-object which converges to $A$, where $B^{\bullet/A}$ is the \v{C}ech nerve of $A \to B$ and $\Tot_{n}(B^{\bullet/A})$ is the limit of the $n$-truncation of $B^{\bullet/A}$.
\end{definition}

\begin{proposition}[{\cite[Corollary 3.33]{Mat16}}]\label{des}
A finitely presented faithfully flat map $A \to B$ of rings is descendable.
\end{proposition}

\begin{theorem}\label{colimit}
Let $A \to B$ be a finitely presented faithfully flat map of rings.
Let $N_{B/A} \in \Dcal(\underline{B})$ be as in Corollary \ref{limit}, and let $N_{B/A} \to \underline{A}$ be the natural map.
Then the induced augmented semisimplicial diagram $\{N_{B/A}^{\otimes (m+1)}\}_{m\geq -1}$ is a colimit diagram (i.e., $\displaystyle \varinjlim_{[m] \in \Delta_s^{\op}}N_{B/A}^{\otimes (m+1)} \overset{\sim}{\to} A$) , where $N_{B/A}^{\otimes (m+1)}$ is the $(m+1)$-fold derived tensor product of $N_{B/A}$ over $A_{\bs}$.
\end{theorem}
\begin{proof}
We may assume that $A$ is a finitely generated $\Zbb$-algebra by a limit argument.
Since $A_{\bs} \to B_{\bs}$ is steady, we have the following equivalence of functors from $\Dcal(A_{\bs})$ to $\Dcal(A_{\bs})$:
\begin{align*}
&- \otimes_{A_{\bs}}^{\Lbb} (B\otimes_A B)_{\bs} \\
\simeq{} &(-\otimes_{A_{\bs}}^{\Lbb} B_{\bs}) \otimes_{A_{\bs}}^{\Lbb} B_{\bs} \\
\simeq{} &R\intHom_{\underline{A}}(N_{B/A}, R\intHom_{\underline{A}}(N_{B/A}, -))\\
\simeq{} &R\intHom_{\underline{A}}(N_{B/A}\otimes_{A_{\bs}}^{\Lbb} N_{B/A},-). 
\end{align*}
By repeating the same argument, we get the following equivalence of functors from $\Dcal(A_{\bs})$ to $\Dcal(A_{\bs})$:
$$- \otimes_{A_{\bs}}^{\Lbb} (B^{m/A})_{\bs} \simeq R\intHom_{\underline{A}}(N_{B/A}^{\otimes (m+1)},-),$$
where $B^{m/A}$ is the $(m+1)$-fold tensor product of $B$ over $A$.
Therefore, the induced augmented semisimplicial diagram $\{N_{B/A}^{\otimes (m+1)}\}_{m\geq -1}$ is equivalent to the underlying augmented semisimplicial diagram of the augmented simplicial diagram $\{R\intHom_{\underline{A}}(\underline{B^{m/A}}, \underline{A})\}_{m \geq -1}$.
Since the inclusion functor $N(\Delta_s^{\op}) \subset N(\Delta^{\op})$ is cofinal by \cite[Lemma 6.5.3.7]{HTT}, the theorem follows from Proposition \ref{des}.
\end{proof}

\begin{corollary}\label{cons}
Let $A\to B$ be a finitely presented faithfully flat map of rings. Then the functor
$$- \otimes_{A_{\bs}}^{\Lbb} B_{\bs} \colon \Dcal(A_{\bs}) \to \Dcal(B_{\bs})$$
is conservative.
\end{corollary}
\begin{proof}
We take an object $M \in \Dcal(A_{\bs})$ such that $M \otimes_{A_{\bs}}^{\Lbb} B_{\bs} \simeq 0$. We want to show that $M$ is equivalent to $0$.
By Theorem \ref{colimit}, we get the following equivalence:
\begin{align*}
M \simeq{} & R\intHom_{\underline{A}}(\underline{A},M)\\
\simeq{} &R\intHom_{\underline{A}}(\varinjlim_{[m] \in \Delta_s^{\op}} N_{B/A}^{\otimes (m+1)}, M)\\
\simeq{} &\varprojlim_{[m] \in \Delta_s}R\intHom_{\underline{A}}(N_{B/A}^{\otimes (m+1)}, M).
\end{align*}
On the other hand, we have 
\begin{align*}
&R\intHom_{\underline{A}}(N_{B/A}^{\otimes (m+1)}, M)\\
\simeq{} &R\intHom_{\underline{A}}(N_{B/A}^{\otimes m}, R\intHom_{\underline{A}}(N_{B/A},M))\\
\simeq{} &R\intHom_{\underline{A}}(N_{B/A}^{\otimes m}, M \otimes_{A_{\bs}}^{\Lbb} B_{\bs})\\

\simeq{} &0
\end{align*}
for every $m \geq 0$.
Therefore, we find $M \simeq 0$.
\end{proof}
 
Faithfully flat descent follows from the following general result.
 
\begin{lemma}\label{general des}
Let $(\Acal, \Mcal) \to (\Bcal, \Ncal)$ be a map in $\AnRing$. We assume that it satisfies the following:
\begin{enumerate}
\item
$(\Acal, \Mcal) \to (\Bcal, \Ncal)$ is steady.
\item
The functor $- \otimes_{(\Acal, \Mcal)}^{\Lbb} (\Bcal, \Ncal) \colon \Dcal(\Acal, \Mcal) \to \Dcal(\Bcal, \Ncal)$ commutes with small limits.
\item
The functor $- \otimes_{(\Acal, \Mcal)}^{\Lbb} (\Bcal, \Ncal) \colon \Dcal(\Acal, \Mcal) \to \Dcal(\Bcal, \Ncal)$ is conservative.
\end{enumerate}
Then the map of $\infty$-categories 
$$ \Dcal(\Acal, \Mcal) \overset{\sim}{\lra} \varprojlim_{[n] \in \Delta} \Dcal((\Bcal, \Ncal)^{n/(\Acal, \Mcal)})$$
is an equivalence, where $(\Bcal, \Ncal)^{\bullet/(\Acal, \Mcal)}$ is the \v{C}ech nerve of $(\Acal, \Mcal) \to (\Bcal, \Ncal)$.
\end{lemma}
\begin{proof}
We will check the conditions (a) and (b) of \cite[Proposition 5.2.2.36]{HA}.
The condition (a) follows from (3). Let us check the condition (b).
We note that the functor 
$$\chi \colon\Delta^{\triangleleft} \to \Cat_{\infty};\; [n] \mapsto \Dcal((\Bcal,\Ncal)^{(n+1)/(\Acal,\Mcal)}),$$ where we denote the cone point of $\Delta^{\triangleleft}$ by $[-1]$, 
classifies a biCartesian fibration $q\colon \Ccal \to \Delta^{\triangleleft}$ since for every $\alpha \colon [m] \to [n]$, $$\chi(\alpha) \colon \Dcal((\Bcal,\Ncal)^{(m+1)/(\Acal,\Mcal)}) \to \Dcal((\Bcal,\Ncal)^{(n+1)/(\Acal,\Mcal)})$$ has a right adjoint functor (\cite[Proposition 4.7.4.17]{HA}).
We take $M \in \Fun_{\Delta^{\triangleleft}}(\Delta, \Ccal)$ which carries each morphism in $\Delta$ to a $q$-coCartesian morphism in $\Ccal$.
We put $$M^n =M([n]) \in q^{-1}([n]) \simeq \Dcal((\Bcal,\Ncal)^{(n+1)/(\Acal,\Mcal)}).$$
By \cite[Corollary 4.3.1.11]{HTT} and its proof, we can extend $M$ to a $q$-limit diagram $\overline{M} \in \Fun_{\Delta^{\triangleleft}}(\Delta^{\triangleleft}, \Ccal)$, and $M([-1]) \in q^{-1}([-1]) \simeq \Dcal(\Acal,\Mcal)$ is equivalent to $\displaystyle \varprojlim_{[n] \in \Delta} M^n$.
We will show that $\overline{M}$ carries each map $[-1] \to [n]$ to a $q$-coCartesian morphism in $\Ccal$.
It is equivalent to show 
$$(\varprojlim_{[n] \in \Delta} M^n)\otimes_{(\Acal,\Mcal)}^{\Lbb}(\Bcal,\Ncal)^{(m+1)/(\Acal,\Mcal)}\simeq M^m.$$
We define a functor $F \colon \Delta^{\triangleleft} \to \Delta ;\; [n] \mapsto [n+1]$ by
$$
F(\alpha)(i)=
\begin{cases}
\alpha(i) & (i\in [n]), \\
m+1 & (i=n+1)
\end{cases}
$$
for $\alpha \colon[n] \to [m]$ in $\Delta^{\triangleleft}$.
Let $\overline{M}^{\prime}\in \Fun(\Delta^{\triangleleft}, \Ccal)$ be the augmented simplicial object given by composing $M$ with the functor $F$, and $M^{\prime}$ be the underling simplicial object of $\overline{M}^{\prime}$.
We regard $\overline{M}$, $\overline{M}^{\prime}$ as augmented simplicial objects of $\Dcal(\Acal,\Mcal)$, $\Dcal(\Bcal,\Ncal)$.
Then by (1), we have an equivalence $M \otimes_{(\Acal,\Mcal)}^{\Lbb}(\Bcal,\Ncal) \simeq M^{\prime}$.
By \cite[Lemma 6.1.3.16]{HTT}), $\overline{M}^{\prime}$ is a limit diagram, that is, $\displaystyle \varprojlim_{[n] \in \Delta} M^{(n+1)} \simeq M^0$.
Therefore, we get $\displaystyle(\varprojlim_{[n] \in \Delta} M^n)\otimes_{(\Acal,\Mcal)}^{\Lbb}(\Bcal,\Ncal)\simeq M^0$ from (2).
We let $f_m \colon (\Bcal, \Ncal) \to (\Bcal, \Ncal)^{(m+1)/(\Acal,\Mcal)}$ denote the morphism of analytic rings induced from the map $[0] \to [m] ;\; 0 \mapsto 0$.
Then we get the following equivalence:
\begin{align*}
&(\varprojlim_{[n] \in \Delta} M^n)\otimes_{(\Acal,\Mcal)}^{\Lbb}(\Bcal,\Ncal)^{(m+1)/(\Acal,\Mcal)}\\
\simeq &((\varprojlim_{[n] \in \Delta} M^n)\otimes_{(\Acal,\Mcal)}^{\Lbb}(\Bcal,\Ncal)) \otimes_{(\Bcal, \Ncal),f_m}^{\Lbb} (\Bcal, \Ncal)^{(m+1)/(\Acal,\Mcal)}\\
\simeq &M^0 \otimes_{(\Bcal, \Ncal),f_m}^{\Lbb} (\Bcal, \Ncal)^{m+1/(\Acal,\Mcal)}\\
\simeq & M^m.
\end{align*}
\end{proof}

\begin{remark}
Since the codiagonal map $(\Bcal, \Ncal) \otimes_{(\Acal, \Mcal)}^{\Lbb} (\Bcal, \Ncal) \to (\Bcal, \Ncal)$ in $\AnRing$ is steady for a steady map $(\Acal, \Mcal) \to (\Bcal, \Ncal)$ (cf.\ \cite[Proposition 2.3.19 (2)]{Mann22}), 
the condition (1) of Lemma \ref{general des} implies that for every map $[m] \to [n]$ in $\Delta$, the map $(\Bcal,\Ncal)^{m/(\Acal,\Mcal)} \to (\Bcal,\Ncal)^{n/(\Acal,\Mcal)}$ is steady.
Therefore Lemma \ref{general des} follows also from the Barr-Beck-Lurie monadicity theorem \cite[4.7.5]{HA}, and we can weaken the condition (2) of Lemma \ref{general des} as stated in loc.\ cit.
\end{remark}

\begin{theorem}\label{dd}
Let $A \to B$ be a finitely presented faithfully flat map of rings.
Let $B^{n/A}$ denote the $(n+1)$-fold tensor product of $B$ over $A$.
Then we have an equivalence of $\infty$-categories 
$$\Dcal(A_{\bs}) \overset{\sim}{\lra} \varprojlim_{[n] \in \Delta} \Dcal((B^{n/A})_{\bs}). $$
\end{theorem}

\begin{proof}
It follows from Proposition \ref{prop:dp}, Proposition \ref{steady2}, Corollary \ref{limit}, Corollary \ref{cons}, and Lemma \ref{general des}.
\end{proof}

Finally, we prove the uniform boundedness of $N_{B/A}$, which is essential for the proof of the main theorem.
\begin{proposition}\label{bdd}
Let $A \to B$ be a finitely presented flat map of rings.
We take $N_{B/A} \in \Dcal(\underline{B})$ as in Corollary \ref{limit}.
Then for every $n \geq 1$, $N_{B/A}^{\otimes n}$ is quasi-isomorphic to a complex of the form $0 \to M^0 \to M^1 \to 0$ where $M^0$ is placed in cohomological degree $0$ and $M^0, M^1$ are projective objects of $\Mod_{A_{\bs}}^{\cond}$.
\end{proposition}
\begin{proof}
By the proof of Theorem \ref{colimit}, it is enough to show the claim in the case where $n=1$.
We take a finitely generated $\Zbb$-subalgebra $A_0$ of $A$ and a finitely presented flat $A_0$-algebra $B_0$ such that $A \otimes_{A_0} B_0 \cong B$.
Then $N_{B/A}$ is quasi-isomorphic to $R\intHom_{\underline{A_0}}(\underline{B_0}, \underline{A_0})\otimes_{(A_0)_\bs}^{\Lbb} A_{\bs}$, so it is enough to show that $R\intHom_{\underline{A_0}}(\underline{B_0}, \underline{A_0})$ is quasi-isomorphic to a complex of the form $\displaystyle 0 \to \prod_I \underline{A_0} \to \prod_J \underline{A_0} \to 0$, where we note that $\displaystyle \prod_I \underline{A_0}, \prod_J \underline{A_0}$ are projective objects of $\Mod_{(A_0)_{\bs}}^{\cond}$.
By \cite[Seconde partie: Corollary 3.3.2]{RG71} $B_0$ has a free resolution of length $1$ over $A_0$, and the claim follows from taking $A_0$-duals of the resolution of $B_0$.
\end{proof}
\section{Adic completeness and small complete adic rings}
In this section, we will introduce some notions which are necessary for proving the main theorem.
\subsection{Adic completeness}
First, we will introduce the notion of adic completeness in the context of analytic rings.
It is an analogue of derived adic completeness of ordinary rings (cf.\ \cite[\href{https://stacks.math.columbia.edu/tag/091N}{Tag 091N}]{stacks-project}).
For our need, it is sufficient to consider analytic rings associated to complete $\pi$-adic rings where $\pi$ is a non-zero-divisor.
For more general theory, see \cite[\S 2.12]{Mann22}.

Let $A$ be a complete $\pi$-adic ring where $\pi \in A$ is a non-zero-divisor.
\begin{definition}
An object $M \in \Dcal(A_{\bs})$ is \textit{$\pi$-adically complete} if $\displaystyle R\varprojlim(\cdots \overset{\pi}{\to}M \overset{\pi}{\to}M\overset{\pi}{\to}M)$ is equivalent to $0$.
We denote the full $\infty$-subcategory of $\pi$-adically complete objects of $\Dcal(A_{\bs})$ by $\hat{\Dcal}_{\pi}(A_{\bs})$.
\end{definition}
By definition,  $\hat{\Dcal}_{\pi}(A_{\bs})$ is a stable subcategory of $\Dcal(A_{\bs})$ stable under retracts.

\begin{proposition}\label{drvd cplt}
Let $M$ be an object of $\Dcal(A_{\bs})$. 
Then the following are equivalent:
\begin{enumerate}
\item
The object $M$ is $\pi$-adically complete.
\item
For every integer $i \in \Zbb$, $H^i(M)$ is $\pi$-adically complete (as an object of $\Dcal(A_{\bs})$).
\item
The natural map $\displaystyle M \to R\varprojlim_{n} (M\otimes_{A_{\bs}}^{\Lbb} \underline{A/\pi^n})$ is an equivalence.
\end{enumerate}
\end{proposition}
\begin{proof}
First, we show the equivalence of (1) and (2).
We have the following fiber sequence:
$$ R\varprojlim(\cdots \overset{\pi}{\to}M \overset{\pi}{\to}M) \to \prod_{\Nbb} M \to \prod_{\Nbb} M,$$
where the map $\displaystyle \prod_{\Nbb} M \to \prod_{\Nbb} M$ is given by $(x_n) \mapsto (x_n-\pi x_{n+1})$.
Since direct products in $\Mod_{A_{\bs}}^{\cond}$ are exact, we have the following long exact sequence:
\begin{align*}
\cdots \to H^{i-1}(R\varprojlim(\cdots \overset{\pi}{\to}M \overset{\pi}{\to}M)) &\to \prod_{\Nbb} H^{i-1}(M) \to \prod_{\Nbb} H^{i-1}(M) \\
\to H^i(R\varprojlim(\cdots \overset{\pi}{\to}M \overset{\pi}{\to}M)) &\to \prod_{\Nbb} H^i(M) \to \prod_{\Nbb} H^i(M) \to \cdots.
\end{align*}
Since $H^i(M)$ is $\pi$-adically complete if and only if the map $\displaystyle\prod_{\Nbb} H^i(M) \to \prod_{\Nbb} H^i(M)$ is an isomorphism, the equivalence of (1) and (2) follows.

Next, we show the equivalence of (1) and (3).
We have the following diagram:
$$\xymatrix{
\cdots \ar[r]^{\pi} &M \ar[r]^{\pi} \ar[d]^{\pi^3} &M \ar[r]^{\pi} \ar[d]^{\pi^2} &M \ar[d]^{\pi} \\
\cdots \ar[r]^{\id} &M \ar[r]^{\id} \ar[d] &M \ar[r]^{\id} \ar[d] &M\ar[d] \\
\cdots \ar[r] &M\otimes_{A_{\bs}}^{\Lbb} \underline{A/\pi^3} \ar[r] &M\otimes_{A_{\bs}}^{\Lbb} \underline{A/\pi^2} \ar[r] &M\otimes_{A_{\bs}}^{\Lbb} \underline{A/\pi},
}
$$
where the vertical sequences are fiber sequences.
Therefore, we have the following fiber sequence:
$$R\varprojlim(\cdots \overset{\pi}{\to}M \overset{\pi}{\to}M) \to M \to R\varprojlim_{n} (M\otimes_{A_{\bs}}^{\Lbb} \underline{A/\pi^n}).$$
The equivalence of (1) and (3) follows from it. 
\end{proof}

\begin{lemma}\label{H0 vanish}
Let $M$ be a $\pi$-adically complete object of $\Dcal(A_{\bs})$.
If $M \otimes_{A_{\bs}}^{\Lbb} \underline{A/\pi}$ belongs to $\Dcal^{\leq -2}(A_{\bs})$, then we have $M \in \Dcal^{\leq -1}(A_{\bs})$.
\end{lemma}
\begin{proof}
We can check $M \otimes_A^{\Lbb} A/\pi^n\in\Dcal^{\leq -2}(A)$ for $n \geq 1$ by induction, 
so the lemma follows from the fact that the cohomological dimension of $\displaystyle R\varprojlim_{n \in \Nbb}$ in $\Mod_{A_{\bs}}^{\cond}$ is 1.
\end{proof}

\subsection{Small complete adic rings}
Next, we define small complete adic rings. 
\begin{definition}
Let $A$ be a complete adic ring.
The topological ring $A$ is said to be \textit{small} if there exists an ideal of definition $I \subset A$, a finitely generated $\Zbb$-algebra $R$ and an integral ring map $R \to A/I$.
\end{definition}

\begin{lemma}
Let $A$ be a complete adic ring.
Then the following are equivalent:
\begin{enumerate}
\item
The ring $A$ is small.
\item
There exists a finitely generated $\Zbb$-algebra $R$ and a ring map $R \to A$ such that, for every ideal of definition $I\subset A$, $R\to A/I$ is integral.
\end{enumerate}
\end{lemma}
\begin{proof}
It follows from a straightforward argument.
\end{proof}

\begin{example}
The ring of integers of a non-archimedean local field is small.
Moreover $\Ocal_{\Cbb_p}$ is small.
\end{example}
\begin{example}
Let $A$ be a complete adic ring and $I \subset A$ be an ideal of definition, and $p \colon A \to A/I$ be the projection map.
Let $\bar{R} \subset A/I$ be a finitely generated $\Zbb$-subalgebra.
Then $R=p^{-1}(\bar{R})$ is a small complete $IR$-adic ring.
\end{example}

The following proposition is the reason why we introduce small complete adic rings.
\begin{proposition}\label{small}
Let $A$ be a small complete $\pi$-adic ring where $\pi \in A$ is a non-zero-divisor.
Then for every extremally disconnected set $S$, $A_{\bs}[S]$ is $\pi$-adically complete.
In particular, every compact object of $\Dcal(A_{\bs})$ is $\pi$-adically complete. 
\end{proposition}
\begin{proof}
We have to show that $\displaystyle A_{\bs}[S] \to R\varprojlim_n (A_{\bs}[S] \otimes_{A_{\bs}}^{\Lbb} \underline{A/\pi^n})$ is an equivalence.
First, we prove that $\displaystyle R\varprojlim_n (A_{\bs}[S] \otimes_{A_{\bs}}^{\Lbb} \underline{A/\pi^n})$ is concentrated on degree $0$.
We have the following equivalence:
\begin{align*}
&A_{\bs}[S]\otimes_{A_{\bs}}^{\Lbb}\underline{A/\pi^n}\\
\simeq{} &((A_{\disc})_{\bs}[S]\otimes_{(A_{\disc})_{\bs}}^{\Lbb}\underline{A})\otimes_{A_{\bs}}^{\Lbb} \underline{A/\pi^n}\\
\simeq{} &(A_{\disc})_{\bs}[S]\otimes_{(A_{\disc})_{\bs}}^{\Lbb}\underline{A/\pi^n}\\
\simeq{} &(A/\pi^n)_{\bs}[S],
\end{align*}
where the first and third equivalence follows from Construction \ref{construction}.
Therefore, for every extremally disconnected set $T$, we find that $\displaystyle R\varprojlim_n (A_{\bs}[S] \otimes_{A_{\bs}}^{\Lbb} \underline{A/\pi^n})(T)$ is concentrated on degree $0$, 
since the transition maps of the projective system $\{(A_{\bs}[S] \otimes_{A_{\bs}}^{\Lbb} \underline{A/\pi^n})(T)\}_{n}$ are surjective.
Hence, $\displaystyle R\varprojlim_n (A_{\bs}[S] \otimes_{A_{\bs}}^{\Lbb} \underline{A/\pi^n})$ is concentrated on degree $0$.

From the above, it suffices to show that 
$$A_{\bs}[S] \to \varprojlim_n (A_{\bs}[S] \otimes_{A_{\bs}} \underline{A/\pi^n})$$
is an isomorphism.
We take a set $I$ and an isomorphism $\displaystyle \Zbb_{\bs}[S]\cong \prod_{I} \underline{\Zbb}$.
Then by Remark \ref{rem:dirprod}, we have the equivalence
$$\varprojlim_n (A_{\bs}[S] \otimes_{A_{\bs}} \underline{A/\pi^n}) \cong \varprojlim_n (\varinjlim_{B_n \subset A/\pi^n} \prod_I \underline{B_n}),$$
where the colimits taken over all finitely generated $\Zbb$-subalgebra $B_n$ of $A/\pi^n$.
From the above and Theorem \ref{dir prod}, we can regard $A_{\bs}[S]$ and $\displaystyle \varprojlim_n (A_{\bs}[S] \otimes_{A_{\bs}}\underline{A/\pi^n})$ as submodules of $\displaystyle \prod_I \underline{A}$.
The inclusion $\displaystyle A_{\bs}[S] \subset \varprojlim_n (A_{\bs}[S] \otimes_{A_{\bs}}\underline{A/\pi^n})$ is clear, so it is enough to show the opposite inclusion.
We take an extremally disconnected set $T$ and an element $\displaystyle (a_i)_{i \in I} \in (\varprojlim_n (A_{\bs}[S] \otimes_{A_{\bs}}\underline{A/\pi^n}))(T)=\varprojlim_n (\varinjlim_{B_n \subset A/\pi^n} \prod_I \underline{B_n}(T))\subset \prod_{I} \underline{A}(T)$.
Let $B_n$ be a finitely generated $\Zbb$-subalgebra of $A/\pi^n$ such that $a_i \bmod \pi^n \in \underline{B_n}(T)$ for every $i\in I$.
We can take a finitely generated $\Zbb$-algebra $R$ and a ring map $R \to A$ such that, for every ideal of definition $I \subset A$, $R \to A/I$ is integral.
Let us construct inductively a finitely generated $R$-submodule $M_n$ of $A/\pi^n$ satisfying that for every $n \geq 1$, $a_i \bmod \pi^n \in \underline{M_n}(T)$ for every $i\in I$ and $M_n \to M_{n-1}$ is well-defined and surjective.
First, we take generators $x_1,\ldots, x_m$ of $B_n$ as a $\Zbb$-algebra. 
As $x_1,\ldots,x_m$ are integral over $R$, the $R$-subalgebra of $A/\pi^n$ generated by $x_1,\ldots,x_m$ is a finitely generated $R$-module.
We take $y_1,\ldots, y_l \in A/\pi^n$ such that the image of them in $A/\pi^{n-1}$ generates $M_{n-1}$.
Let $M_n^{\prime}$ be the $R$-submodule of $A/\pi^n$ generated by $x_1,\ldots,x_m,y_1,\ldots,y_l$, and let $M_n$ be the preimage of $M_{n-1}$ under the map $M_n^{\prime} \to A/\pi^{n-1}$.
Then $M_n$ satisfies the conditions, where we note that $M_n$ is finitely generated $R$-module since $R$ is noetherian.
Let $M$ be the limit of $(M_n)$.
By the construction, $M$ is quasi-finitely generated over $R$ and $\displaystyle(a_i)_{i\in I} \in \prod_I \underline{M}(T)$.
Therefore, we get $\displaystyle(a_i) \in \varinjlim_{B \subset A^+, M} \prod_I \underline{M}(T) =A_{\bs}[S](T)$, which proves the former part of the proposition.

The latter part follows from \cite[Lemma A.2.1]{Mann22} and the fact that $\hat{\Dcal}_{\pi}(A_{\bs})$ is a full $\infty$-subcategory of $\Dcal(A_{\bs})$ stable under finite colimits and retracts.
\end{proof}

\begin{lemma}\label{filtered colimit}
Let $A$ be a complete $I$-adic ring, where $I \subset A$ is a finitely generated ideal.
We write $A/I$ as a filtered colimit $\displaystyle \varinjlim_{\lambda \in \Lambda} \bar{A_{\lambda}}$ where each $\bar{A_{\lambda}}$ is a finitely generated $\Zbb$-subalgebra of $A/I$.
Let $A_{\lambda}$ be the preimage of $\bar{A_{\lambda}}$ under $A\to A/I$, which is a small complete $I$-adic ring.
Then in $\AnRing$, we have the following equivalence:
$$A_{\bs} \simeq \varinjlim_{\lambda \in \Lambda} (A_{\lambda})_{\bs}.$$
\end{lemma}
\begin{proof}
It follows from a straightforward computation by using Theorem \ref{dir prod}.
\end{proof}

This lemma is useful because of the following.
For an $\infty$-category $\Ccal$, we denote the full $\infty$-subcategory of compact objects of $\Ccal$ by $\Ccal_{\cpt}$.
\begin{lemma}[{\cite[Lemma 2.7.4]{Mann22}}]\label{limit argument}
Let $\displaystyle \Acal = \varinjlim_{\lambda \in \Lambda} \Acal_{\lambda}$ be a filtered colimit in $\AnRing$.
Then the following map induced by the scalar extension functors is an equivalence:
$$ \varinjlim_{\lambda \in \Lambda} \Dcal(\Acal_{\lambda})_{\cpt} \to \Dcal(\Acal)_{\cpt}.$$
\end{lemma}

 
\section{Faithfully flat descent over a complete non-archimedean field}
In this section, we will prove faithfully flat descent for affinoid pairs over a complete non-archimedean field. 

Let $K$ be a complete non-archimedean field, and let $\Ocal_K$ denote the ring of integers of $K$, and $\pi$ denote a pseudo-uniformizer of $K$.
We begin with a review of classical results.
 \begin{definition}
 \begin{enumerate}
 \item
 A \textit{topologically finitely presented $\Ocal_K$-algebra} is a topological $\Ocal_K$-algebra which is a topological quotient of $\Ocal_K\langle T_1, \ldots, T_n \rangle$ for some $n$ by a finitely generated ideal, where $\Ocal_K\langle T_1, \ldots, T_n \rangle$ is equipped with the $\pi$-adic topology.
An \textit{admissible $\Ocal_K$-algebra} is a $\pi$-torsion-free topologically finitely presented $\Ocal_K$-algebra.
 \item
 An \textit{affinoid $K$-algebra} is a topological $K$-algebra which is a quotient topological ring of $K\langle T_1, \ldots, T_n \rangle$ for some $n$ by an ideal.
 \end{enumerate}
 \end{definition}
 
 \begin{remark}
 \begin{enumerate}
 \item
 For a topologically finitely presented $\Ocal_K$-algebra $A$, the topology of $A$ coincides with the $\pi$-adic topology of $A$.
 \item 
 For a topologically finitely presented $\Ocal_K$-algebra $A$, $A[1/\pi]$ naturally becomes an affinoid $K$-algebra.
 \item
Affinoid $K$-algebras are noetherian. 
\item
Topologically finitely presented $\Ocal_K$-algebras are not necessarily noetherian but coherent (see \cite[Proposition 1.3]{FRG1}).
 \end{enumerate}
 \end{remark}
 
 \begin{definition}
 A map $A\to B$ of topologically finitely presented $\Ocal_K$-algebras (resp.\  affinoid $K$-algebras) is flat (resp.\  faithfully flat) if it is flat (resp.\  faithfully flat) as a map of ordinary rings.
 \end{definition}
 
\begin{lemma}[{\cite[Theorem 4.1]{FRG2}}]\label{flattening}
Let $A\to B$ be a faithfully flat map of affinoid $K$-algebras. 
Then there exists a rational covering $\{\Spa(A_i, (A_i)^{\circ}) \to \Spa(A,A^{\circ}) \}$ which satisfies that for each $i$ there exists a faithfully flat map $A_i^{\prime} \to B_i^{\prime}$ of admissible $\Ocal_K$-algebras such that $A_i^{\prime}[1/\pi] \to B_i^{\prime}[1/\pi]$ is isomorphic to $A_i \to B \hotimes_{A} A_i$, where $ \hotimes$ is the completed tensor product.
\end{lemma} 

\begin{lemma}\label{formal model}
Let $A$ be a topologically finitely presented $\Ocal_K$-algebra.
Then we have the following equivalence in $\AnRing$:
$$(A[1/\pi], A[1/\pi]^{\circ})_{\bs} \simeq A_{\bs} \otimes_{{\Ocal_K}_{\bs}}^{\Lbb} (K, \Ocal_K)_{\bs}.$$
\end{lemma}
\begin{proof}
We check the conditions (1) and (2) of Lemma \ref{pushout}. 
As $\underline{A[1/\pi]}$ is $(K, \Ocal_K)_{\bs}$-complete, we get the natural map $\underline{A} \otimes_{{\Ocal_K}_{\bs}}^{\Lbb} (K, \Ocal_K)_{\bs} \to \underline{A[1/\pi]}$, and we want to prove that it is an equivalence.
Since $\underline{A} \otimes_{\underline{\Ocal_K}}^{\Lbb} \underline{K}$ is equivalent to the colimit of the system in $\Dcal(\underline{\Ocal_K})$
$$\underline{A} \overset{\pi}{\to}\underline{A} \overset{\pi}{\to}\underline{A} \to \cdots,$$
$\underline{A} \otimes_{\underline{\Ocal_K}}^{\Lbb} \underline{K}$ is equivalent to $\underline{A[1/\pi]}$, and it is $(K, \Ocal_K)_{\bs}$-complete.
Therefore, we have $\underline{A[1/\pi]} \simeq \underline{A} \otimes_{\underline{\Ocal_K}}^{\Lbb} \underline{K} \simeq \underline{A} \otimes_{{\Ocal_K}_{\bs}}^{\Lbb} (K, \Ocal_K)_{\bs}$, which proves (1).
Since $A[1/\pi]^{\circ}$ is the integral closure of $A$ in $A[1/\pi]$, (2) follows from Proposition \ref{prop:int}.
\end{proof}

\begin{lemma}
Let $A$ be a topologically finitely presented $\Ocal_K$-algebra, and $M$ be a perfect $A$-module with the $\pi$-adic topology. 
Then $\underline{M}$ is quasi-isomorphic to a bounded complex of finite projective $\underline{A}$-modules.
\end{lemma}
\begin{proof}
It easily follows from \cite[Lemma 3.1]{And21}.
\end{proof}  
 
The following is an analogue of Lemma \ref{disc rp}. 
\begin{lemma}\label{rp}
Let $A \to B$ be a flat map of admissible $\Ocal_K$-algebras or a flat map of affinoid $K$-algebras. Then there exists a decomposition 
$$A \to A\langle T_1,\ldots,T_n\rangle \to B$$
such that $B$ is a perfect $A\langle T_1,\ldots,T_n\rangle$-module.
\end{lemma}
\begin{proof}
We will prove it when $A \to B$ is a flat map of topologically finitely presented $\Ocal_K$-algebras. 
We can also prove the other case by almost the same way.
We take a surjection $A\langle T_1,\ldots,T_n\rangle \to B$, and we will show that $B$ is a perfect $A\langle T_1,\ldots,T_n\rangle$-module.
We take any maximal ideal $\mfrak$ of $A\langle T_1,\ldots,T_n\rangle$. Since $\pi$ is topologically nilpotent, we have $\pi \in \mfrak$.
Since the ring $(A/\pi)/{\sqrt{(0)}}$ is finitely generated over the residue field of $\Ocal_K$, $A/\pi$ is a Jacobson ring.
Therefore, the prime ideal $\nfrak = \mfrak \cap A$ of $A$ is a maximal ideal, and the global dimension of $A/\nfrak[T_1,\ldots,T_n]$ is equal to $n$.
Since $A,B$ are $\pi$-torsion-free, we have the following equivalence:
\begin{align*}
& B \otimes_{A\langle T_1,\ldots,T_n\rangle}^{\Lbb} A/\nfrak[T_1,\ldots,T_n]\\
\simeq{} &(B \otimes_{A\langle T_1,\ldots,T_n\rangle}^{\Lbb} A/\pi[T_1,\ldots,T_n])\otimes_{A/\pi[T_1,\ldots,T_n]}^{\Lbb} A/\nfrak[T_1,\ldots,T_n]\\
\simeq{} &B/\pi \otimes_{A/\pi[T_1,\ldots,T_n]}^{\Lbb} (A/\pi[T_1,\ldots,T_n]\otimes_{A/\pi}^{\Lbb} A/\nfrak) \\
\simeq{} &B/\pi \otimes_{A/\pi}^{\Lbb} A/\nfrak \\
\simeq{} &B/{\nfrak B}.
\end{align*}
Therefore, we get the following equivalence:
\begin{align*}
&B \otimes_{A\langle T_1,\ldots,T_n\rangle}^{\Lbb} A\langle T_1,\ldots,T_n\rangle/\mfrak \\
\simeq{} & (B \otimes_{A\langle T_1,\ldots,T_n\rangle}^{\Lbb} A/\nfrak[T_1,\ldots,T_n]) \otimes_{A/\nfrak[T_1,\ldots,T_n]}^{\Lbb} A\langle T_1,\ldots,T_n\rangle/\mfrak \\
\simeq{} &B/{\nfrak B} \otimes_{A/\nfrak[T_1,\ldots,T_n]}^{\Lbb} A\langle T_1,\ldots,T_n\rangle/\mfrak.
\end{align*}
Therefore, we have $H^i(B \otimes_{A\langle T_1,\ldots,T_n\rangle}^{\Lbb} A\langle T_1,\ldots,T_n\rangle/\mfrak)=0$ for all $i \notin [-n, 0]$.
Since $A\langle T_1,\ldots,T_n\rangle$ is coherent, $B$ is a pseudo-coherent $A\langle T_1,\ldots,T_n\rangle$-module, 
we get the claim by Lemma \ref{perfect}.
\end{proof}

\begin{corollary}\label{flat bc}
Let $A\to B$ and $A \to C$ be flat maps of admissible $\Ocal_K$-algebras.
Then we have the following equivalence:
$$(B \hotimes_A C)_{\bs} \simeq B_{\bs} \otimes_{A_{\bs}}^{\Lbb} C_{\bs}.$$
Moreover, a similar statement holds true for affinoid $K$-algebras.
\end{corollary}
\begin{proof}
By Lemma \ref{rp}, we may assume that $B$ is equal to $A\langle T_1,\ldots, T_n\rangle$ or $B$ is a perfect $A$-module.
If $B$ is equal to $A\langle T_1,\ldots, T_n\rangle$, then the claim follows from Lemma \ref{tensor} and Lemma \ref{poly bs}.
If $B$ is a perfect $A$-module, then the claim follows from Lemma \ref{pushout}.
\end{proof}

By using Lemma \ref{rp}, we can prove the following analogue of Proposition \ref{steady2}.

\begin{proposition}\label{steady3}
Let $A \to B$ be a flat map of admissible $\Ocal_K$-algebras. 
Then $A_{\bs} \to B_{\bs}$ is steady.
Moreover, a similar statement holds true for affinoid $K$-algebras.
\end{proposition} 
\begin{proof}
The same argument as in the proof of Proposition \ref{steady2} works well.
\end{proof}

We can prove the following analogue of Corollary \ref{limit}.

\begin{theorem}\label{limit2}
Let $A \to B$ be a flat map of admissible $\Ocal_K$-algebras. 
Then there exists an object $N_{B/A} \in \Dcal(\underline{B})$ satisfying the following conditions:
\begin{itemize}
\item
The object $N_{B/A}$ is $A_{\bs}$-complete and compact as an object of $\Dcal(A_{\bs})$.
\item
We have an equivalence of functors from $\Dcal(A_{\bs})$ to $\Dcal(B_{\bs})$
$$-\otimes_{A_{\bs}}^{\Lbb} B_{\bs} \overset{\sim}{\lra} R\intHom_{\underline{A}}(N_{B/A}, -).$$
\end{itemize}
In particular, the functor
$$- \otimes_{A_{\bs}}^{\Lbb} B_{\bs} \colon \Dcal(A_{\bs}) \to \Dcal(B_{\bs})$$
commutes with small limits.
\end{theorem}

\begin{proof}
By the same argument as in the proof of Theorem \ref{key1}, we may assume that $B$ is equal to $A\langle T_1,\ldots ,T_n \rangle$ or $B$ is a perfect $A$-module.
If $B$ is a perfect $A$-module, then the same argument as in Theorem \ref{key1} works well.
We assume $B=A\langle T_1,\ldots ,T_n \rangle$.
We define $N_{A\langle T_1,\ldots ,T_n \rangle/A} = R\intHom_{\underline{\Zbb}}(\Zbb[T_1,\ldots,T_n],\Zbb) \otimes_{\Zbb_{\bs}}^{\Lbb} A_{\bs}$, which is an object of $\Dcal(\underline{A\langle T_1,\ldots ,T_n \rangle})$.
Then we can show that it satisfies the conditions by the same argument as in Corollary \ref{limit}.
\end{proof}

\begin{remark}
The functoriality result as in Remark \ref{functoriality} holds true.
It can be proved by the same way.
\end{remark}

\begin{lemma}\label{semisimplicial}
Let $\Dcal$ be a stable $\infty$-category, and $X_{\bullet}$ be a semisimplicial object of $\Dcal$.
Then for every $n \geq 1$, we have the following fiber sequence:
$$\varinjlim_{[m] \in \Delta_{s,\leq n-1}^{\op}}X_m \to \varinjlim_{[m] \in \Delta_{s,\leq n}^{\op}}X_m \to X_n[n].$$
\end{lemma}
\begin{proof}
By \cite[Remark 4.3.4]{DAG8}, we have the following fiber sequence:
$$\varinjlim_{[m] \in \Delta_{s,\leq n-1}^{\op}}X_m \to \varinjlim_{[m] \in \Delta_{s,\leq n}^{\op}}X_m \to \cofib(X_n\otimes \partial \Delta^n \to X_n).$$
Therefore, the claim follows from the equivalence $\cofib(X_n\otimes \partial \Delta^n \to X_n) \simeq X_n[n].$
\end{proof}

\begin{theorem}\label{rig dual des}
Let $f \colon A \to B$ be a faithfully flat map of admissible $\Ocal_K$-algebras. 
Let $N_{B/A} \in \Dcal(\underline{B})$ be as in Theorem \ref{limit2} and $N_{B/A} \to \underline{A}$ be the natural map in $\Dcal(A_{\bs})$. 
Then the map $\displaystyle \varinjlim_{[m] \in \Delta_{s,\leq 2}^{\op}} N_{B/A}^{\otimes (m+1)} \to \underline{A}$ has a section, where $N_{B/A}^{\otimes (m+1)}$ is the $(m+1)$-fold derived tensor product of $N_{B/A}$ over $A_{\bs}$.
\end{theorem}
\begin{proof}
We take a directed system $(A_{\lambda})_{\lambda \in \Lambda}$ of small complete $\pi$-adic subrings of $A$ as in Lemma \ref{filtered colimit}.
Since $N_{B/A}$ is compact as an object of $\Dcal(A_{\bs})$, we can take $\lambda \in \Lambda$ and $N^{\lambda}_{B/A} \to \underline{A_{\lambda}}$ in $\Dcal((A_{\lambda})_{\bs})_{\cpt}$ such that 
$$(N^{\lambda}_{B/A} \to \underline{A_{\lambda}}) \otimes_{(A_{\lambda})_{\bs}}^{\Lbb} A_{\bs} \simeq (N_{B/A} \to \underline{A})$$ in $\Dcal(A_{\bs})$ by Lemma \ref{limit argument}.
By enlarging $\lambda$, if necessary, we may assume that there exists a finitely presented faithfully flat $A_{\lambda}/\pi$-algebra $S_{\lambda}$ such that $S_{\lambda} \otimes_{A_{\lambda}/\pi} A/\pi \cong B/\pi$.
Then, the two maps 
\begin{equation}\label{map1}
(N^{\lambda}_{B/A} \to \underline{A_{\lambda}}) \otimes_{(A_{\lambda})_{\bs}}^{\Lbb} (A_{\lambda}/\pi)_{\bs}
\end{equation}
and 
\begin{equation}\label{map2}
N_{S_{\lambda}/(A_{\lambda}/\pi)} \to \underline{A_{\lambda}/\pi}
\end{equation}
in $\Dcal((A_{\lambda}/\pi)_{\bs})_{\cpt}$ become equivalent after applying $-\otimes_{(A_{\lambda}/\pi)_{\bs}}^{\Lbb} (A/\pi)_{\bs}$.
Therefore, we may assume that two maps (\ref{map1}) and (\ref{map2}) are equivalent by enlarging $\lambda$, if necessary. 
We put $$\displaystyle N_n\coloneqq \varinjlim_{[m] \in \Delta_{s,\leq n}^{\op}}(N^{\lambda}_{B/A})^{\otimes (m+1)}  \otimes_{(A_{\lambda})_{\bs}}^{\Lbb} (A_{\lambda}/\pi)_{\bs}.$$
By Lemma \ref{semisimplicial}, we have a fiber sequence 
$$N_n \to N_{n+1} \to ((N^{\lambda}_{B/A})^{\otimes (n+2)}  \otimes_{(A_{\lambda})_{\bs}}^{\Lbb} (A_{\lambda}/\pi)_{\bs})[n+1].$$
Since $(N^{\lambda}_{B/A})^{\otimes (n+2)}  \otimes_{(A_{\lambda})_{\bs}}^{\Lbb} (A_{\lambda}/\pi)_{\bs}$ lies in $\Dcal^{[0,1]}((A_{\lambda}/\pi)_{\bs})$ by Proposition \ref{bdd},
we get that $\tau^{\geq2-n}(N_n) \to \tau^{\geq2-n}(N_{n+1})$ is an equivalence.
By Theorem \ref{colimit}, we have $\displaystyle \varinjlim_{n} N_n \simeq \underline{A_{\lambda}/\pi}$, so we get an equivalence $\tau^{\geq0} N_2 \simeq \underline{A_{\lambda}/\pi}$.
Since $H^{-1}(\underline{A_{\lambda}/\pi})=0$, we get $\cofib(N_2 \to \underline{A_{\lambda}/\pi}) \in \Dcal^{\leq -2}((A_{\lambda}/\pi)_{\bs})$.
We put $\displaystyle N_2^{\prime} \coloneqq \varinjlim_{[m] \in \Delta_{s,\leq 2}^{\op}}(N^{\lambda}_{B/A})^{\otimes (m+1)}$, which is a compact object of $\Dcal((A_{\lambda})_{\bs})$ and therefore $\pi$-adically complete by Proposition \ref{small}.
We have $$\cofib(N_2^{\prime} \to \underline{A_{\lambda}})\otimes_{(A_{\lambda})_{\bs}}^{\Lbb} (A_{\lambda}/\pi)_{\bs} \simeq \cofib(N_2 \to \underline{A_{\lambda}/\pi}),$$ so we get $\cofib(N_2^{\prime} \to \underline{A_{\lambda}}))\in \Dcal^{\leq -1}((A_{\lambda})_{\bs})$ by Lemma \ref{H0 vanish}.
We have the following fiber sequences:
$$\xymatrix{
N_2^{\prime} \ar[r] &\underline{A_{\lambda}} \ar[r] & \cofib(N_2^{\prime} \to \underline{A_{\lambda}})  \\
N^{\lambda}_{B/A} \ar[r]\ar[u] &\underline{A_{\lambda}} \ar[r]_-{\alpha}\ar[u]_{\id} &\cofib(N^{\lambda}_{B/A} \to \underline{A_{\lambda}}) \ar[u]_{\beta}.
}
$$
By $\cofib(N_2^{\prime} \to \underline{A_{\lambda}}))\in \Dcal^{\leq -1}((A_{\lambda})_{\bs})$, $\beta\circ\alpha$ is zero, where we note that $\underline{A_{\lambda}}$ is a projective object in $\Mod_{(A_{\lambda})_{\bs}}^{\cond}$, so $N_2^{\prime} \to \underline{A_{\lambda}}$ has a section.
Therefore, $$\displaystyle (\varinjlim_{[m] \in \Delta_{s,\leq 2}^{\op}} N_{B/A}^{\otimes (m+1)} \to \underline{A}) \simeq (N_2^{\prime} \to \underline{A_{\lambda}}) \otimes_{(A_{\lambda})_{\bs}}^{\Lbb} A_{\bs}$$ also has a section.
\end{proof}

\begin{remark}
By the similar argument as in \cite[Proposition 3.20]{Mat16}, we can deduce from Theorem \ref{rig dual des} that the ind-object $\displaystyle \{\varinjlim_{[m] \in \Delta_{s,\leq n}^{\op}} N_{B/A}^{\otimes (m+1)}\}_n$ of $\Dcal(A_{\bs})$ is a constant ind-object which converges to $\underline{A}$, but we will not use it, so we omit the details.
\end{remark}

\begin{corollary}\label{cons2}
Let $f \colon A \to B$ be a faithfully flat map of admissible $\Ocal_K$-algebras. 
Then, the functor
$$- \otimes_{A_{\bs}}^{\Lbb} B_{\bs} \colon \Dcal(A_{\bs}) \to \Dcal(B_{\bs})$$
is conservative.
\end{corollary}
\begin{proof}
We take an object $M \in \Dcal(A_{\bs})$ such that $M \otimes_{A_{\bs}}^{\Lbb} B_{\bs} \simeq 0$. We want to show that $M$ is equivalent to $0$.
Since $M \simeq R\intHom_{\underline{A}}(\underline{A},M)$ is a retract of 
$$R\intHom_{\underline{A}}(\varinjlim_{[m] \in \Delta_{s,\leq 2}^{\op}} N_{B/A}^{\otimes (m+1)}, M)\simeq\varprojlim_{[m] \in \Delta_{s,\leq 2}}R\intHom_{\underline{A}}(N_{B/A}^{\otimes (m+1)}, M),$$
it is enough to show that $R\intHom_{\underline{A}}(N_{B/A}^{\otimes (m+1)}, M)\simeq 0$, which can be shown by the same way as in Corollary \ref{cons}.
\end{proof}

\begin{remark}
If $K$ is a non-archimedean local field, then we can prove the equivalence $N_{B/A} \simeq R\intHom_{\underline{A}}(\underline{B}, \underline{A})$ by the same argument as in the proof of Theorem \ref{key1}.
Therefore, the proofs become easier by using the results in \cite{Mat22}.
\end{remark}

From Theorem \ref{rig dual des}, we get our main theorems.

\begin{theorem}\label{intmain}
Let $f \colon A \to B$ be a faithfully flat map of admissible $\Ocal_K$-algebras. 
Let $B^{n/A}$ denote the $(n+1)$-fold completed tensor product of $B$ over $A$.
Then we have an equivalence of $\infty$-categories 
$$\Dcal(A_{\bs}) \overset{\sim}{\lra} \varprojlim_{[n] \in \Delta} \Dcal((B^{n/A})_{\bs}). $$
\end{theorem}

\begin{proof}
It follows from Lemma \ref{general des}, Corollary \ref{flat bc}, Proposition \ref{steady3}, Theorem \ref{limit2}, and Corollary \ref{cons2}.
\end{proof}

\begin{theorem}\label{main theorem}
Let $f \colon A \to B$ be a faithfully flat map of affinoid $K$-algebras. 
Let $B^{n/A}$ denote the $(n+1)$-fold completed tensor product of $B$ over $A$.
Then we have an equivalence of $\infty$-categories 
$$\Dcal((A,A^{\circ})_{\bs}) \overset{\sim}{\lra} \varprojlim_{[n] \in \Delta} \Dcal((B^{n/A}, (B^{n/A})^{\circ})_{\bs}). $$
\end{theorem}

\begin{proof}
By Theorem \ref{andes} and Lemma \ref{flattening}, we may assume that there exists a faithfully flat map $A^{\prime} \to B^{\prime}$ of admissible $\Ocal_K$-algebras such that $A^{\prime}[1/\pi] \to B^{\prime}[1/\pi]$ is isomorphic to $A \to B$. 
Noting that there is an equivalence $(B, B^{\circ})_{\bs} \simeq (A, A^{\circ})_{\bs} \otimes_{A^{\prime}_{\bs}}^{\Lbb} B^{\prime}_{\bs}$ by Lemma \ref{formal model}, we can easily check the conditions of  Lemma \ref{general des} as in Theorem \ref{intmain}.
\end{proof} 

 
\section{Descent of pseudo-coherent complexes and perfect complexes}
In this section, we will recover descent for pseudo-coherent complexes and perfect complexes from the main theorem by the method used in \cite[Section 5]{And21}. 
Let us begin by briefly recalling the notion used in \cite[Section 5]{And21}.
\begin{definition}
Let $A$ be a discrete ring. 
\begin{enumerate}
\item
An object of $\Dcal(A)$ is called a \textit{pseudo-coherent complex} if the following are satisfied:
\begin{itemize}
\item
$M$ is bounded above.
\item
For each integer $i$, the functor $\Ext_{A}^0(M,-) \colon \Dcal^{\geq i}(A) \to \Ab$ commutes with small filtered colimits.
\end{itemize}
We denote the full $\infty$-subcategory of pseudo-coherent complexes in $\Dcal(A)$ by $\PCoh(A)$.
In addition, for every integer $n$, we denote the full $\infty$-subcategory of pseudo-coherent complexes in $\Dcal^{\leq n}(A)$ by $\PCoh^{\leq n}(A)$.
\item
An object of $\Dcal(A)$ is called a \textit{perfect complex} if it is quasi-isomorphic to a complex of the form 
$$ 0 \to P_{n+m} \to \cdots \to P_{n+1} \to P_{n} \to 0,$$
where $P_{i}$ are finite projective $A$-modules.
We denote the full $\infty$-subcategory of perfect complexes in $\Dcal(A)$ by $\Perf(A)$.
In addition, for every pair of integers $a \leq b$, we denote the full $\infty$-subcategory of perfect complexes with Tor-amplitude in $[a,b]$ by $\Perf^{[a,b]}(A)$.
\end{enumerate}
\end{definition}

\begin{remark}
The category $\FP(A)$ of finite projective $A$-modules is equivalent to $\Perf^{[0,0]}(A)$.
\end{remark}

\begin{remark}[{\cite[Proposition 5.14]{And21}}]
An object of $\Dcal(A)$ is a pseudo-coherent complex if and only if it is quasi-isomorphic to a complex of the form 
$$ \cdots \to P_{n+m} \to \cdots \to P_{n+1} \to P_{n} \to 0,$$
where $P_{i}$ are finite projective $A$-modules.
\end{remark}

Let $(\Acal, \Mcal)$ be a (0-truncated) analytic ring.  

\begin{definition}[{\cite[Definition 5.13]{And21}}]
An object $M \in \Dcal(\Acal, \Mcal)$ is called a \textit{pseudo-coherent complex} if the following are satisfied:
\begin{itemize}
\item
$M$ is bounded above.
\item
For each integer $i$, the functor $\Ext_{\Acal}^0(M,-) \colon \Dcal^{\geq i}(\Acal, \Mcal) \to \Ab$ commutes with small filtered colimits.
\end{itemize}
We denote the full $\infty$-subcategory of pseudo-coherent complexes in $\Dcal(\Acal, \Mcal)$ by $\PCoh(\Acal, \Mcal)$.
\end{definition}

\begin{lemma}\label{pcoh0}
Let $M$ be an object of $\PCoh(\Acal, \Mcal)$. 
Then for each integer $i$, the functor $R\Hom_{\underline{A}}(M,-) \colon \Dcal^{\geq i}(\Acal, \Mcal) \to \Dcal(\Ab)$ commutes with small filtered colimits.
\end{lemma}
\begin{proof}
It easily follows from the definition.
\end{proof}

\begin{lemma}\label{pcoh}
An object $M \in \Dcal(\Acal, \Mcal)$ is a pseudo-coherent complex if and only if $M$ is quasi-isomorphic to a bounded above complex with terms of the form $\Mcal[S]$ for various extremally disconnected sets $S$.
\end{lemma}
\begin{proof}
We note that $\displaystyle\bigoplus_{i=1}^n \Mcal[S_i]$ is isomorphic to $\displaystyle\Mcal[\coprod_{i=1}^n S_i]$ for extremally disconnected sets $S_i$, and that $\displaystyle\coprod_{i=1}^n S_i$ is an extremally disconnected set.
Then the lemma follows from \cite[Proposition 5.14]{And21}.
\end{proof}

\begin{definition}[{\cite[Definition 5.26]{And21}}]
Let $M$ be an object of $\Dcal(\Acal, \Mcal)$.
Then the object $R\intHom_{\Acal}(M, \Acal) \in \Dcal(\Acal, \Mcal)$ is called the \textit{dual} of $M$ and denoted by $M^{\lor}$.
\end{definition}

\begin{definition}[{\cite[Definition 5.30]{And21}}]
An object $M\in \Dcal(\Acal, \Mcal)$ is said to be \textit{nuclear} if for any extremally disconnected set $S$, the natural map 
$$(\Mcal[S]^{\lor}\otimes_{(\Acal, \Mcal)}^{\Lbb} M)(\ast) \to M(S)$$
is an equivalence in $\Dcal(\Ab)$.
We denote the full $\infty$-subcategory of nuclear objects in $\Dcal(\Acal, \Mcal)$ by $\Dcal(\Acal, \Mcal)_{\nc}$.
\end{definition}

Next, we define a functor which associates $\underline{A}$-modules to $A$-modules.
 
\begin{definition}[{\cite[Definition 5.8, Theorem 5.9]{And21}}]
Let $(A,A^+)$ be an analytic complete affinoid pair.
Then the \textit{condensification functor} is the fully faithful functor $\Cond_{(A,A^+)} \colon \Dcal(A) \to \Dcal((A,A^+)_{\bs})$ given by $$M \mapsto \dCond_{A_{\disc}}(M) \otimes_{(A_{\disc}, A^+_{\disc})_{\bs}}^{\Lbb} (A,A^+)_{\bs}.$$
An object of $\Dcal((A,A^+)_{\bs})$ is said to be discrete if it lies in the essential image of $\Cond_{(A,A^+)}$.
\end{definition}
 
\begin{lemma}[{\cite[Lemma 5.10]{And21}}]\label{functorial}
Let $(A,A^+) \to (B,B^+)$ be a map of analytic complete affinoid pairs.
Then the following diagram naturally commutes:
$$\xymatrix@C=60pt{
\Dcal(A) \ar[r]^-{\Cond_{(A,A^+)}} \ar[d]^{-\otimes_{A}^{\Lbb}B} &\Dcal((A,A^+)_{\bs}) \ar[d]^{-\otimes_{(A,A^+)_{\bs}}^{\Lbb}(B,B^+)_{\bs}}\\
\Dcal(B) \ar[r]^-{\Cond_{(B,B^+)}} &\Dcal((B,B^+)_{\bs}) 
.}
$$
\end{lemma} 

\begin{lemma}[{\cite[Theorem 5.9, Proposition 5.22, Lemma 5.45, Lemma 5.46, Lemma 5.47, Theorem 5.50]{And21}}]\label{andmain}
Let $(A, A^+)$ be a sheafy analytic complete affinoid pair. Then the functor $$\Cond_{(A,A^+)} \colon \Dcal(A) \to \Dcal((A,A^+)_{\bs})$$ induces the following equivalences
\begin{itemize}
\item
$\PCoh(A) \overset{\sim}{\lra} \PCoh((A,A^+)_{\bs})_{\nc}.$
\item
$\Perf(A) \overset{\sim}{\lra} \PCoh((A,A^+)_{\bs})_{\nc, \cpt}$,
\end{itemize}
where $\PCoh((A,A^+)_{\bs})_{\nc}$ is the full $\infty$-subcategory consisting of nuclear objects of $\Dcal((A,A^+)_{\bs})$ in $\PCoh((A,A^+)_{\bs})$, and $\PCoh((A,A^+)_{\bs})_{\nc, \cpt}$ is the full $\infty$-subcategory consisting of nuclear compact objects of $\Dcal((A,A^+)_{\bs})$ in $\PCoh((A,A^+)_{\bs})$.
Moreover, for an object $M\in \PCoh(A)$ and each $i\in \Zbb$, $H^i(M)=0$ if and only if $H^i(\Cond_{(A,A^+)}(M))=0$.
\end{lemma}

The following is an analogue of \cite[Proposition 5.38]{And21}.

\begin{lemma}\label{rhom}
Let $K$ be a complete non-archimedean field and $A \to B$ be a flat map of affinoid $K$-algebras.
Then for every extremally disconnected set $S$ and every object $M \in \Dcal((A, A^+)_{\bs})$, we have an equivalence
\begin{align*}
&R\intHom_{\underline{A}}((A, A^+)_{\bs}[S], M) \otimes_{(A, A^+)_{\bs}}^{\Lbb} (B,B^+)_{\bs} \\
\simeq{} &R\intHom_{\underline{B}}((B,B^+)_{\bs}[S], M \otimes_{(A, A^+)_{\bs}}^{\Lbb} (B,B^+)_{\bs}).
\end{align*}
\end{lemma}
\begin{proof}
We take $N_{B/A} \in \Dcal(\underline{B})$ as in Theorem \ref{limit2}.
Then the lemma follows from the following computation:
\begin{align*}
&R\intHom_{\underline{A}}((A, A^+)_{\bs}[S], M) \otimes_{(A, A^+)_{\bs}}^{\Lbb} (B,B^+)_{\bs} \\
\simeq{} &R\intHom_{\underline{A}}(N_{B/A}, R\intHom_{\underline{A}}((A, A^+)_{\bs}[S], M)) \\
\simeq{} &R\intHom_{\underline{A}}(N_{B/A}\otimes_{(A,A^+)_{\bs}}^{\Lbb} (A, A^+)_{\bs}[S], M) \\
\simeq{} &R\intHom_{\underline{A}}((A, A^+)_{\bs}[S], R\intHom_{\underline{A}}(N_{B/A}, M)) \\
\simeq{} &R\intHom_{\underline{A}}((A, A^+)_{\bs}[S], M \otimes_{(A, A^+)_{\bs}}^{\Lbb} (B,B^+)_{\bs}) \\
\simeq{} &R\intHom_{\underline{B}}((B,B^+)_{\bs}[S], M \otimes_{(A, A^+)_{\bs}}^{\Lbb} (B,B^+)_{\bs}).
\end{align*} 
\end{proof}

The following is an analogue of \cite[Proposition 5.39]{And21}.

\begin{lemma}\label{tordim}
Let $K$ be a complete non-archimedean field and $A \to B$ be a flat map of affinoid $K$-algebras.
Then there exists a non-negative integer $k$ such that for every integer $i$ and every $M \in \Dcal^{\geq i}((A,A^{\circ})_{\bs})$, $M\otimes_{(A,A^{\circ})_{\bs}}^{\Lbb} (B,B^{\circ})_{\bs}$ lies in $\Dcal^{\geq i-k}((A,A^{\circ})_{\bs})$.
\end{lemma}
\begin{proof}
By Lemma \ref{rp}, we may assume that $B$ is equal to $A\langle T \rangle$ or $B$ is a perfect $A$-module (if $B$ is a perfect $A$-module then $B$ could not be flat over $A$, but it causes no problem).
The latter case is trivial, so we may assume that $B=A\langle T \rangle$.
In this case, the claim follows from Lemma \ref{poly bs}, Lemma \ref{poly} and that $R\intHom_{\underline{\Zbb}}(\underline{\Zbb[T]}, \underline{\Zbb})$ is a compact projective $\Zbb_{\bs}$-module.
\end{proof}

By Lemma \ref{andmain}, in order to prove descent for perfect complexes and pseudo-coherent complexes, it is enough to prove descent for nuclear objects, pseudo-coherent objects, and compact objects.
\begin{theorem}\label{somedes}
Let $K$ be a complete non-archimedean field and $A \to B$ be a faithfully flat map of affinoid $K$-algebras. 
Let $B^{n/A}$ denote the $(n+1)$-fold completed tensor product of $B$ over $A$.
Then we have the following equivalences of $\infty$-categories.
\begin{enumerate}
\item
$\displaystyle\PCoh((A,A^{\circ})_{\bs}) \overset{\sim}{\lra} \varprojlim_{[n] \in \Delta} \PCoh((B^{n/A}, (B^{n/A})^{\circ})_{\bs})$.
\item
$\displaystyle\Dcal((A,A^{\circ})_{\bs})_{\nc} \overset{\sim}{\lra} \varprojlim_{[n] \in \Delta} \Dcal((B^{n/A}, (B^{n/A})^{\circ})_{\bs})_{\nc}$.
\item
$\displaystyle\Dcal((A,A^{\circ})_{\bs})_{\cpt} \overset{\sim}{\lra} \varprojlim_{[n] \in \Delta} \Dcal((B^{n/A}, (B^{n/A})^{\circ})_{\bs})_{\cpt}$.
\end{enumerate}
\end{theorem}
\begin{proof}
The same argument as in \cite[Theorem 5.42]{And21} works well to prove (2) by using Lemma \ref{rhom}  instead of \cite[Proposition 5.38]{And21}.

By Lemma \ref{flattening} and \cite[Theorem 5.41, Theorem 5.44]{And21}, in order to prove (1) and (3), we may assume that there exists a faithfully flat map $A^{\prime} \to B^{\prime}$ of admissible $\Ocal_K$-algebras such that $A^{\prime}[1/\pi] \to B^{\prime}[1/\pi]$ is isomorphic to $A \to B$.
Let $N_{B^{\prime}/A^{\prime}}$ be as in Theorem \ref{limit2}.
By Theorem \ref{rig dual des}, we have a retraction $\displaystyle\varprojlim_{[m] \in \Delta_{s,\leq 2}} (M \otimes_{(A,A^{\circ})_{\bs}}^{\Lbb} (B^{\bullet/A}, (B^{\bullet/A})^{\circ})_{\bs}) \to M$ of the natural map $\displaystyle M \to \varprojlim_{[m] \in \Delta_{s,\leq n}} (M \otimes_{(A,A^{\circ})_{\bs}}^{\Lbb} (B^{\bullet/A}, (B^{\bullet/A})^{\circ})_{\bs})$ which is functorial in $M \in \Dcal((A,A^{\circ})_{\bs})$.

We prove (1). 
It is enough to show the following:
\begin{itemize}
\item
For $M \in \PCoh((A,A^{\circ})_{\bs})$, $M \otimes_{(A,A^{\circ})_{\bs}}^{\Lbb} (B^{m/A},(B^{m/A})^{\circ})_{\bs}$ is pseudo-coherent for every non-negative integer $m$.
\item
For $M \in \Dcal((A,A^{\circ})_{\bs})$, if $M \otimes_{(A,A^{\circ})_{\bs}}^{\Lbb} (B^{m/A},(B^{m/A})^{\circ})_{\bs}$ is pseudo-coherent for every non-negative integer $m$, then $M$ is pseudo-coherent.
\end{itemize}
The former follows from Lemma \ref{pcoh}. Let $i$ be a non-negative integer and $J \to \Dcal^{\geq i}((A,A^{\circ})_{\bs});\; j \mapsto M_j$ be any small filtered diagram. 
We will show that the map 
\begin{align}
\varinjlim_J(R\Hom_{\underline{A}}(M,M_j))\to R\Hom_{\underline{A}}(M, \varinjlim_J M_j)
\label{eq4}
\end{align}
is an equivalence.
It is a retract of the map
\begin{align*}
&\varinjlim_J(R\Hom_{\underline{A}}(M,\varprojlim_{[m] \in \Delta_{s,\leq n}}(M_j\otimes_{(A,A^{\circ})_{\bs}}^{\Lbb} (B^{m/A}, (B^{m/A})^{\circ})_{\bs})))\\
\to &R\Hom_{\underline{A}}(M, \varprojlim_{[m] \in \Delta_{s,\leq n}}((\varinjlim_J M_j)\otimes_{(A,A^{\circ})_{\bs}}^{\Lbb} (B^{m/A}, (B^{m/A})^{\circ})_{\bs}))).
\end{align*}
Therefore, it is enough to show that the above map is an equivalence.
It follows from the following computation:
\begin{align*}
&\varinjlim_J(R\Hom_{\underline{A}}(M,\varprojlim_{[m] \in \Delta_{s,\leq n}}(M_j\otimes_{(A,A^{\circ})_{\bs}}^{\Lbb} (B^{m/A}, (B^{m/A})^{\circ})_{\bs})))\\
\simeq{} &\varprojlim_{[m] \in \Delta_{s,\leq n}}\varinjlim_J(R\Hom_{\underline{A}}(M,M_j\otimes_{(A,A^{\circ})_{\bs}}^{\Lbb} (B^{m/A}, (B^{m/A})^{\circ})_{\bs}))\\
\simeq{} &\varprojlim_{[m] \in \Delta_{s,\leq n}}\varinjlim_J(R\Hom_{\underline{B^{m/A}}}(M\otimes_{(A,A^{\circ})_{\bs}}^{\Lbb} (B^{m/A}, (B^{m/A})^{\circ})_{\bs},M_j\otimes_{(A,A^{\circ})_{\bs}}^{\Lbb} (B^{m/A}, (B^{m/A})^{\circ})_{\bs}))\\
\simeq{} &\varprojlim_{[m] \in \Delta_{s,\leq n}}(R\Hom_{\underline{B^{m/A}}}(M\otimes_{(A,A^{\circ})_{\bs}}^{\Lbb} (B^{m/A}, (B^{m/A})^{\circ})_{\bs},\\
&\qquad\qquad\qquad\qquad \varinjlim_J(M_j\otimes_{(A,A^{\circ})_{\bs}}^{\Lbb} (B^{m/A}, (B^{m/A})^{\circ})_{\bs})))\\
\simeq{} & \varprojlim_{[m] \in \Delta_{s,\leq n}}(R\Hom_{\underline{A}}(M,\varinjlim_J(M_j\otimes_{(A,A^{\circ})_{\bs}}^{\Lbb} (B^{m/A}, (B^{m/A})^{\circ})_{\bs})))\\
\simeq{} & \varprojlim_{[m] \in \Delta_{s,\leq n}}(R\Hom_{\underline{A}}(M,(\varinjlim_J M_j)\otimes_{(A,A^{\circ})_{\bs}}^{\Lbb} (B^{m/A}, (B^{m/A})^{\circ})_{\bs}))\\
\simeq{} & R\Hom_{\underline{A}}(M, \varprojlim_{[m] \in \Delta_{s,\leq n}}((\varinjlim_J M_j)\otimes_{(A,A^{\circ})_{\bs}}^{\Lbb} (B^{m/A}, (B^{m/A})^{\circ})_{\bs})),
\end{align*}
where the first equivalence follows from the commutativity of finite limits and filtered colimits, and the third equivalence follows from Lemma \ref{pcoh0}, pseudo-coherence of $M\otimes_{(A,A^{\circ})_{\bs}}^{\Lbb} (B^{m/A}, (B^{m/A})^{\circ})_{\bs}$, and the fact that $\{M_j\otimes_{(A,A^{\circ})_{\bs}}^{\Lbb} (B^{m/A}, (B^{m/A})^{\circ})_{\bs}\}_j$ is uniformly bounded below by Lemma \ref{tordim}.

Finally we prove (3). 
It is enough to show the following:
\begin{itemize}
\item
For $M \in \Dcal((A,A^{\circ})_{\bs})_{\cpt}$, $M \otimes_{(A,A^{\circ})_{\bs}}^{\Lbb} (B^{m/A},(B^{m/A})^{\circ})_{\bs}$ is compact for every non-negative integer $m$.
\item
For $M \in \Dcal((A,A^{\circ})_{\bs})$, if $M \otimes_{(A,A^{\circ})_{\bs}}^{\Lbb} (B^{m/A},(B^{m/A})^{\circ})_{\bs}$ is compact for every non-negative integer $m$, then $M$ is compact.
\end{itemize}
The former can be proven by the same argument as in \cite[Theorem 5.41]{And21}, and the latter can be proven by the same argument of the latter part of the proof of (1).
\end{proof}

\begin{theorem}
Let $K$ be a complete non-archimedean field and $A \to B$ be a faithfully flat map of affinoid $K$-algebras. 
Let $B^{n/A}$ denote the $(n+1)$-fold completed tensor product of $B$ over $A$.
Then we have the following equivalence of $\infty$-categories:
$$\Ccal(A) \overset{\sim}{\lra} \varprojlim_{[n] \in \Delta} \Ccal(B^{n/A}),$$
where $\Ccal$ is one of the following $\infty$-categories: $\PCoh$, $\PCoh^{\leq m}$, $\Perf$, $\Perf^{[a,b]}$, and $\FP$.
\end{theorem}
\begin{proof}
The descent for $\PCoh$ and $\Perf$ follows from Theorem \ref{main theorem}, Lemma \ref{andmain}, Theorem \ref{somedes}.
To prove the descent for $\PCoh^{\leq m}$, it is enough to show that for $M \in \PCoh(A)$ if $M \otimes_A^{\Lbb} B$ belongs to $\PCoh^{\leq m}(B)$ then $M$ also belongs to $\PCoh^{\leq m}(A)$, where we implicitly use Lemma \ref{functorial} and Lemma \ref{andmain}.
This follows from the fact that $A \to B$ is faithfully flat.
Finally, to prove the descent for $\Perf^{[a,b]}$, it is enough to show that for $M \in \Perf(A)$ if $M \otimes_A^{\Lbb} B$ belongs to $\Perf^{[a,b]}(B)$ then $M$ also belongs to $\Perf^{[a,b]}(A)$. 
Since $A\to B$ is faithfully flat, we have that for $N\in \Dcal(A)$ and $i\in \Zbb$, $H^i(N)=0$ is equivalent to $H^i(N\otimes_A^{\Lbb}B)=0$. 
The claim easily follows from it.
\end{proof}

\nocite{HTT}
\nocite{FRG1}
\bibliographystyle{my_amsalpha}
\bibliography{bibliography}

\newcommand{\etalchar}[1]{$^{#1}$}
\providecommand{\bysame}{\leavevmode\hbox to3em{\hrulefill}\thinspace}
\providecommand{\MR}{\relax\ifhmode\unskip\space\fi MR }
\providecommand{\MRhref}[2]{%
  \href{http://www.ams.org/mathscinet-getitem?mr=#1}{#2}
}
\providecommand{\href}[2]{#2}
\begin{thebibliography}{FGI{\etalchar{+}}05}

\bibitem[And21]{And21}
Grigory Andreychev, \emph{Pseudocoherent and perfect complexes and vector
  bundles on analytic adic spaces}, arXiv preprint arXiv:
  \href{https://arxiv.org/abs/2105.12591}{2105.12591} (2021).

\bibitem[BG98]{BG98}
Siegfried Bosch and Ulrich G\"{o}rtz, \emph{Coherent modules and their descent
  on relative rigid spaces}, J. Reine Angew. Math. \textbf{495} (1998),
  119--134.

\bibitem[BH19]{BH19}
Clark Barwick and Peter Haine, \emph{Pyknotic objects, {I}. {B}asic notions},
  arXiv preprint arXiv: \href{https://arxiv.org/abs/1904.09966}{1904.09966}
  (2019).

\bibitem[BL93a]{FRG1}
Siegfried Bosch and Werner L\"{u}tkebohmert, \emph{Formal and rigid geometry.
  {I}. {R}igid spaces}, Math. Ann. \textbf{295} (1993), no.~2, 291--317.

\bibitem[BL93b]{FRG2}
\bysame, \emph{Formal and rigid geometry. {II}. {F}lattening techniques}, Math.
  Ann. \textbf{296} (1993), no.~3, 403--429.

\bibitem[BS22]{BS19}
Bhargav Bhatt and Peter Scholze, \emph{{Prisms and prismatic cohomology}},
  Annals of Mathematics \textbf{196} (2022), no.~3, 1135 -- 1275.

\bibitem[Con03]{Con03}
Brian Conrad, \emph{Descent for coherent sheaves on rigid-analytic spaces},
  \url{http://math.stanford.edu/~conrad/papers/cohdescent.pdf}, 2003.

\bibitem[CS19]{CS19}
Kestutis Cesnavicius and Peter Scholze, \emph{Purity for flat cohomology},
  arXiv preprint arXiv: \href{https://arxiv.org/abs/1912.10932}{1912.10932}
  (2019).

\bibitem[FGI{\etalchar{+}}05]{FGA}
Barbara Fantechi, Lothar G\"{o}ttsche, Luc Illusie, Steven~L. Kleiman, Nitin
  Nitsure, and Angelo Vistoli, \emph{Fundamental algebraic geometry:
  {G}rothendieck's {FGA} explained}, Mathematical Surveys and Monographs, vol.
  123, American Mathematical Society, Providence, RI, 2005.

\bibitem[Lur09]{HTT}
Jacob Lurie, \emph{Higher topos theory}, Annals of Mathematics Studies, vol.
  170, Princeton University Press, Princeton, NJ, 2009.

\bibitem[Lur11]{DAG8}
\bysame, \emph{Quasi-coherent sheaves and {T}annaka duality theorems},
  \url{https://www.math.ias.edu/~lurie/papers/DAG-VIII.pdf}, 2011.

\bibitem[Lur17]{HA}
\bysame, \emph{Higher algebra},
  \url{https://www.math.ias.edu/~lurie/papers/HA.pdf}, 2017.

\bibitem[Lur18]{SAG}
\bysame, \emph{Spectral algebraic geometry},
  \url{https://www.math.ias.edu/~lurie/papers/SAG-rootfile.pdf}, 2018.

\bibitem[Man22]{Mann22}
Lucas Mann, \emph{A p-adic 6-functor formalism in rigid-analytic geometry},
  Ph.D. thesis, Rheinische Friedrich-Wilhelms-Universit{\"a}t Bonn, August
  2022, \url{https://hdl.handle.net/20.500.11811/10125}.

\bibitem[Mat16]{Mat16}
Akhil Mathew, \emph{The {G}alois group of a stable homotopy theory}, Adv. Math.
  \textbf{291} (2016), 403--541.

\bibitem[Mat22]{Mat22}
Akhil Mathew, \emph{Faithfully flat descent of almost perfect complexes in
  rigid geometry}, Journal of Pure and Applied Algebra \textbf{226} (2022),
  no.~5, 106938.

\bibitem[RG71]{RG71}
Michel Raynaud and Laurent Gruson, \emph{Crit\`eres de platitude et de
  projectivit\'{e}. {T}echniques de ``platification'' d'un module}, Invent.
  Math. \textbf{13} (1971), 1--89.

\bibitem[Sch19]{CM}
Peter Scholze, \emph{Lectures on condensed mathematics},
  \url{https://people.mpim-bonn.mpg.de/scholze/Condensed.pdf}, 2019.

\bibitem[Sch20]{AG}
\bysame, \emph{Lectures on analytic geometry},
  \url{https://people.mpim-bonn.mpg.de/scholze/Analytic.pdf}, 2020.

\bibitem[SP]{stacks-project}
The Stacks~Project Authors, \emph{Stacks project},
  \url{https://stacks.math.columbia.edu}, 2018.

\end{thebibliography}
\end{document}